\documentclass{amsart}
\usepackage{amssymb}
\usepackage[mathscr]{eucal}
\usepackage{eufrak}
\usepackage{xspace}
\usepackage{mathtools}
\usepackage{color}

\newcommand{\R}{\mathbb{R}}

\newcommand{\DD}{\mathcal{D}}
\newcommand{\Z}{\mathbb{Z}}
\newcommand{\N}{\mathbb{N}}
\newcommand{\C}{\mathbb{C}}

\newcommand{\LL}{\mathcal{L}}

\newcommand{\F}{{\mathcal F}}

\newcommand{\po}{\partial}

\newcommand{\wto}{\rightharpoonup}
\newcommand{\ve}{\varepsilon}

\newcommand{\la}{\langle}
\newcommand{\ra}{\rangle}

\newcommand{\loc}{{\text{\rm loc}}}
\newcommand{\X}{\times}

\renewcommand{\O}{{\mathbf O}}
\renewcommand{\d}{\delta}
\renewcommand{\l}{\lambda}

\renewcommand{\a}{\alpha}
\renewcommand{\b}{\beta}

\newcommand{\s}{\sigma}
\newcommand{\g}{\gamma}
\newcommand{\z}{\zeta}
\renewcommand{\k}{\kappa}
\newcommand{\sgn}{\text{\rm sgn}}

\newcommand{\Om}{\Omega}
\newcommand{\om}{\omega}

\newcommand{\supp}{\text{\rm supp}\,}

\newcommand{\M}{{\mathcal M}}

\renewcommand{\div}{\text{\rm div}\,}

\renewcommand{\supp}{\text{\rm supp}\,}

\newcommand{\cO}{{\mathcal O}}

\newcommand{\cL}{{\mathcal L}}

\newcommand{\cP}{{\mathcal P}}

\newcommand{\bff}{{\mathbf f}}
\newcommand{\bfa}{{\mathbf a}}

\newcommand{\fbf}{\mathbf{f}}

\newcommand{\Abf}{\mathbf{A}}
\newcommand{\abf}{\mathbf{a}}

\newcommand{\Lip}{\text{\rm Lip}}

\newcommand{\B}{{\mathcal B}}

\newcommand{\bbE}{{\mathbb E}}
\newcommand{\bbT}{{\mathbb T}}
\newcommand{\bbP}{{\mathbb P}}
\newcommand{\bbG}{{\mathbb G}}

\newcommand{\BAP}{\operatorname{BAP}}
\newcommand{\AP}{\operatorname{AP}}

\newcommand{\BUC}{\operatorname{BUC}}

\renewcommand{\Lip}{\text{Lip\,}}

\newcommand{\intl}{\int\limits}
\newcommand{\iintl}{\iint\limits}

\newcommand{\ff}{\mathfrak{f}}

\newcommand{\mm}{\mathfrak{m}}

\newcommand{\G}{{\mathcal G}}

\newcommand{\Me}{\operatorname{M}}
\newcommand{\Sp}{\operatorname{Sp}}
\newcommand{\Gr}{\operatorname{Gr}}

\newcommand{\Medint}{\mkern12mu\mbox{\vrule height4pt
         depth-3.2pt
          width5pt}\mkern-16.5mu\int\nolimits}

\theoremstyle{plain}
\newtheorem{theorem}{Theorem}[section]
\newtheorem{corollary}{Corollary}[section]

\newtheorem{lemma}{Lemma}[section]
\newtheorem{proposition}{Proposition}[section]
\theoremstyle{definition}
\newtheorem{definition}{Definition}[section]
\theoremstyle{remark}

\newtheorem{remark}{Remark}[section]
\numberwithin{equation}{section}

\begin{document}

\title[Invariants measures for stochastic conservation laws]
{Invariant measures for stochastic\\ parabolic-hyperbolic equations\\  in the space of 
 almost periodic functions: Lipschitz flux case}

\author[C. Espitia]{Claudia Espitia}
\thanks{C.~Espitia thankfully acknowledges the support from CNPq, through grant proc.\ 140268/2019-7}
\address{Instituto de Matem\'atica Pura e Aplicada - IMPA\\
         Estrada Dona Castorina, 110 \\
         Rio de Janeiro, RJ 22460-320, Brazil}
\email{claudia.duarte@impa.br}

\author[H.~Frid]{Hermano Frid}
\thanks{H.~Frid gratefully acknowledges the support from CNPq, through grant proc.\ 305097/2019-9, and FAPERJ, through grant proc.\ E-26/200.958/2021.}
\address{Instituto de Matem\'atica Pura e Aplicada - IMPA\\
         Estrada Dona Castorina, 110 \\
         Rio de Janeiro, RJ 22460-320, Brazil}
\email{hermano@impa.br}

\author[D.~Marroquin]{Daniel Marroquin}
\thanks{D.~Marroquin thankfully acknowledges the support from CNPq, through grant proc.\ 150118/2018-0.}

\address{Instituto de Matem\'{a}tica - Universidade Federal do Rio de Janeiro\\
Cidade Universit\'{a}ria, 21945-970, Rio de Janeiro, Brazil}
\email{marroquin@im.ufrj.br}

\date{}

\begin{abstract} We study the well-posedness and the long-time behavior of almost periodic solutions to
stochastic degenerate parabolic-hyperbolic equations in any space dimension, under the assumption 
of Lipschitz continuity of the flux  and viscosity functions and a non-degeneracy condition. We show the existence and uniqueness of an invariant measure in a separable subspace of the space of   Besicovitch almost periodic functions.
\end{abstract}

\maketitle

\section{Introduction}\label{S:1}
We  study the well-posedness of the Cauchy problem and the existence and uniqueness of  invariant measures for  stochastic nonlinear degenerate parabolic-hyperbolic equations in the space of Besicovitch almost periodic
functions. Namely, we consider an equation of the form 
\begin{equation}\label{e1.1}
du+\div(\bff(u))\,dt-D^2:\Abf(u)\,dt=\Phi\,d W(t), 
\end{equation}
where  $\bff=(f_1,\cdots,f_N) \in C^2\cap \Lip(\R;\R^N)$,   $\Abf\in C^2\cap\Lip(\R;\R^{N\X N})$, such that $\Abf(u)=(A_{ij}(u))$,  is a symmetric $N \X N$ matrix with  $\Abf'(u):=\abf(u)$  symmetric nonnegative, and we denote $D^2: \Abf(u):=\sum_{i,j=1}^N\po_{x_i x_j}^2 A_{ij}(u)$.    The initial function is given as 
\begin{equation}\label{e1.2}
u(0,x)=u_0(x),\quad x\in\R^N.
\end{equation} 

We recall that the space of real-valued almost periodic functions in $\R^N$, $\AP(\R^N)$,  is the closure in $C_b(\R^N)$, endowed with the $\sup$-norm, of the finite linear combinations of the trigonometric functions $\cos 2\pi\l\cdot x$ and $\sin 2\pi\l\cdot x$, $\l \in\R^N$, or, equivalently, the real part of the closure in $C_b(\R^N;\C)$ of the complex space spanned by  $\{ e^{i2\pi\l\cdot x}\,:\, \l\in\R^N\}$. It is well known (see, e.g.,  \cite{B}) that   
$\AP(\R^N)$ is a sub-algebra of the space of bounded uniformly continuous functions $\BUC(\R^N)$, whose elements
$g$ possess a mean value $\Me(g)$ which is a number such that $g(\ve^{-1}x)\wto \Me(g)$, in the sense of the weak star convergence in $L^\infty(\R^N)$,  and  can be defined by
$$
\Me(g):= \lim_{R\to\infty} R^{-N}\int_{C_R}g(x)\,dx,
$$
where
$$
C_R:=\{x=(x_1,\cdots,x_N)\in\R^N\,:\, |x_i|\le R/2\}
$$
and it holds $\Me(g)=\Me(g(\cdot+\l))$, for any $\l\in\R^N$. We also denote
$$
\Me(g)=\Medint_{\R^N} g(x)\,dx.
$$
Given $g\in\AP(\R^N)$, the spectrum of $g$, $\Sp(g)$, is defined by
$$
\Sp(g):=\{\l\in\R^N\,:\, a_\l:=\Medint_{\R^N} e^{-i2\pi \l\cdot x}g(x)\,dx\ne0\}.
$$   
 It is a well known fact that $\Sp(g)$ is a countable set, which follows easily from Bessel's inequality when we introduce in $\AP(\R^N)$ the inner product 
 $$
 \la g,h\ra:=\Me(gh).
 $$
For $g\in\AP(\R^N)$, we denote by $\Gr(g)$ the smallest additive group containing $\Sp(g)$. The Besicovitch space $\B^1(\R^N)$ is defined as the abstract completion of $\AP(\R^N)$ by the norm
$$
N_1(u):=\sup_{R\to\infty}\frac1{R^N}\int_{C_R}|u(x)|\,dx=\Medint_{\R^N}|u(x)|\,dx,
$$ 
A classical procedure going back to Besicovitch (see \cite{B}) shows that any $g\in\B^1(\R^N)$ has a representative in $L_\loc^1(\R^N)$. 
 
It is also a well known fact  that $\AP(\R^N)$ is isometrically isomorphic with the space $C(\bbG_N)$, where $\bbG_N$ is the so called Bohr compact which is a compact topological group (see, e.g., \cite{DS,Loo}). The Haar measure on $\bbG_N$, $\mm$, such that $\mm(\bbG_N)=1$, is the measure induced by the  mean value $g\mapsto \Me(g)$ defined for all $g\in\AP(\R^N)\sim C(\bbG_N)$.  The topology 
in $\bbG_N$ is generated by the images through the referred isomorphism  (also called Gelfand transforms) of the functions $e^{-i2\pi\l\cdot x}$. The isometric isomorphism between $\AP(\R^N)$ and $C(\bbG_N)$ extends to an isometry between $\B^1(\R^N)$, with the norm given by $N_1$, and $L^1(\bbG_N)$.  
Here and elsewhere in what follows, although we are mainly dealing with real functions,  we switch freely between the real and  the complex version of $\AP(\R^N)$ whenever we want to take advantage of the fact that the latter is generated by the complex exponentials  $e^{-i2\pi\l\cdot x}$.  The translations $\tau_y:\R^N\to\R^N$, $\tau_yx=x+y$,
$y\in\R^N$, extend as homeomorphisms $\tau_y:\bbG_N\to\bbG_N$. Therefore, we can define directional derivatives $D_y g(x)$ of functions  $g\in C(\bbG_N)$ at a point $x\in\bbG_N$, for $y\in\R^N$, $|y|=1$,  by the usual formula, $D_yg(x):=\lim_{|h|\to0} (g(x+hy)-g(x))/h$ whenever the limit exists. In particular, when $y=e_i$, where $e_i$ is the $i$-th element of the canonical basis,
we get the partial derivatives $D_i g(x)$ or $\po_{x_i}g(x)$, or yet $\po g(x)/\po x_i$, $i=1,\cdots,N$.  We then denote by $C^\ell(\bbG_N)$ the space of functions in $C(\bbG_N)$ whose derivatives up to order $k$ also belong to $C(\bbG_N)$, for $\ell\in\N\cup\{0,\infty\}$. It is easy to see that $C^\ell(\bbG_N)$ is isometrically isomorphic with $\AP^\ell(\R^N)$, where the latter is the subspace of $\AP(\R^N)$ whose derivatives up to the order $\ell$ are in $\AP(\R^N)$, for  $\ell\in\N\cup\{0,\infty\}$.

 As usual, if  $(\Om, \F, \F_t, \bbP)$ is a stochastic basis, where $(\F_t)_{t\ge0}$ is a complete filtration,  $W$ is a cylindrical Wiener process, $W(t)=\sum_{k\ge1}\b_k(t) e_k$, where $\b_k(t)$ are independent Brownian processes with respect to the filtration $(\F_t)$, and $\{e_k\}_{k\ge1}$ is a complete orthonormal system in a Hilbert space $H$.  The map $\Phi: H\to L^2(\bbG_N)$ is defined by $\Phi e_k=g_k$ where $g_k\in C^2(\bbG_N)\sim \AP^2(\R^N)$. We assume that there exists a sequence of positive numbers $(\a_k)_{k\ge1}$ satisfying $D_0:=\sum_{k\ge1}\a_k^2<\infty$ such that 
\begin{equation}\label{e1.3-0}
|g_k(x)|+|\nabla_x g_k(x)|+|D^2g_k(x)|\le \a_k,\quad \forall x\in\R^N.
\end{equation}
Observe that from \eqref{e1.3-0} it follows 
\begin{align}
& G^2(x)=\sum_{k\ge1}|g_k(x)|^2\le D_0, \label{e1.3}\\
&\sum_{k\ge1}|g_k(x )-g_k(y)|^2\le D_0 |x-y|^2, \label{e1.4}
\end{align}
for all $x,y\in \R^N$.

 Existence and uniqueness of a kinetic solution to the initial value problem for \eqref{e1.1} in the periodic case is proved by Debussche, Homanov\'a and Vovelle in  \cite{DHV}, for initial data in $\bigcap_{p\ge1}L^p(\bbT^N)$  and by Gess and Homanov\'a in \cite{GH}, for initial data in $L^1(\bbT^N)$. Moreover, both \cite{DHV} and \cite{DV} show  that, given two
initial data $u_0^1,u_0^2\in L^1(\bbT^N)$, the corresponding kinetic solutions satisfy the following contraction property:
\begin{equation}\label{e1.4P}
\|u^1(t)-u^2(t)\|_{L^1(\bbT^N)}\le \|u_0^1-u_0^2\|_{L^1(\bbT^N)}, \ a.s.
\end{equation}
This allows the definition of the transition semigroup in ${\mathcal B}_b(L^1(\bbT^N))$, the latter being the space of bounded Borel functions on $L^1(\bbT^N)$, by
$$
P_t\phi(u_0)= \bbE(\phi(u(t))). 
$$
The existence and uniqueness of an invariant measure with respect to $P_t$ in $L^1(\bbT^N)$ has been proven by Debussche and Vovelle in \cite{DV2} for the case of stochastic conservation laws and extended by Chen and Pang in 
\cite{CPa} for the case of stochastic degenerate parabolic-hyperbolic equations, both under the assumption that the noise coefficients (and, without loss of generality, also the initial data) have zero spatial mean-value.
  
Here, we first address the well-posedness of  Besicovitch almost periodic ($\BAP$, for short) entropy solutions of \eqref{e1.1}-\eqref{e1.2}. This notion is motivated by  Kim's idea in \cite{KJ} of 
defining the new dependent variable $w=u-J$, with $J(t,x):=\sum_{k=1}^\infty g_k(x)\b_k(t)$, and transforming \eqref{e1.1}--\eqref{e1.2} into a deterministic problem for each fixed 
$\om\in\O$.  Based on a Kruzhkov-type inequality established by Chen and Karlsen in \cite{CK}, we prove that the $\BAP$-entropy solutions satisfy a $L^1$-mean semi-contraction property, that is, given two 
$\BAP$-entropy solutions, $u,v$, we have, a.s.,
$$
\Medint_{\R^N}|u(t,x)-v(t,x)|\,dx\le C(t) \Medint_{\R^N}|u_0(x)-v_0(x)|\,dx,
$$ 
for some constant $C(t)$ depending on $t$, $\fbf, \Abf$, and possibly also on $\om$.  
We then define $L^1(\bbG_N)$-entropy solutions as a natural extension of $\BAP$-entropy solutions,  through the isometric isomorphism $\B^1(\R^N)\sim L^1(\bbG_N)$,    and show the existence of such solutions as a consequence of the existence of $\BAP$-entropy solutions, using the semi-contraction property above. Unfortunately, we cannot assert the uniqueness of $L^1(\bbG_N)$-entropy solutions, in general. Nevertheless, we prove that given two $L^1(\bbG_N)$-entropy solutions $u,v$, both obtained
as limits in $L^1(\Om; L^\infty([0,T];L^1(\bbG_N)))$  of $\BAP$-entropy solutions,   we have, a.s., 
$$
\int_{\bbG_N}|u(t,z)-v(t,z)|\,d\mm(z)\le C(t)\int_{\bbG_N}|u_0(z)-v_0(z)|\,d\mm(z),
$$
for the same constant $C(t)$ as above. Here we call $L^1(\bbG_N)$-semigroup solution an $L^1(\bbG_N)$-entropy solution  obtained  as limit in $L^1(\Om; L^\infty([0,T];L^1(\bbG_N)))$ of a sequence of $\BAP$-entropy solutions. The existence of such solutions is also proved here.
So, differently from the hyperbolic case analyzed in \cite{EFM}, in the present degenerate parabolic-hyperbolic case we do not have a proper contraction, and so  this does not allow us  in principle to define a contractive  transition semigroup as in the periodic case.   However, when we restrict ourselves to  a separable subspace of $\AP(\R^N)$, $\AP_*(\R^N)$, whose elements have spectrum contained in a fixed  finitely generated additive group, it has been  shown in \cite{EFM} that there is an  isometric isomorphism between $\B_*^1(\R^N)$, the corresponding Besicovitch space, 
and $L^1(\bbT^P)$, where $P$ is the cardinality of the set of generators of the additive group containing
the spectrum of the functions in $\AP_*(\R^N)$. Let $\bbG_{*N}$ be the compact associated with the algebra $\AP_*(\R^N)$, which is a finitely generated topological subgroup of $\bbG_N$, so $\AP_*(\R^N)\sim C(\bbG_{*N})$, and $\B_*^1(\R^N)\sim L^1(\bbG_{*N})$. Using also the idea of reduction to the periodic case introduced by Panov in \cite{Pv}, it then follows  the contraction property for $L^1(G_{*N})$-semigroup solutions, namely, given any two such solutions $u,v$, a.s., it holds
$$
\int_{\bbG_{*N}}|u(t,z)-v(t,z)|\,d\mm(z)\le \int_{\bbG_{*N}}|u_0(z)-v_0(z)|\,d\mm(z).
$$
Using this contraction property, we can then define the contractive transition semigroup as  in the periodic case in \cite{DV2,CPa} and prove the existence and uniqueness of an invariant measure, provided we assume a suitable non-degeneracy condition (see \eqref{e5.NDC1} and \eqref{e5.NDC}), and also assuming, as in \cite{CPa,DV2}, that the noise coefficients have zero spatial mean-value, that is, 
\begin{equation}
\Medint_{\mathbb{R}^N}g_k(x)\, dx = 0,\quad \text{for all $k\ge 1$.}
\end{equation}

\subsection{Main results}
 The purpose of this paper is to extend the results in \cite{DV2} and  \cite{CPa} to a more general class of oscillatory solutions, at least in the case of Lipschitz flux and viscosity functions. 
 
 Our main result concerning the well posedness of \eqref{e1.1}--\eqref{e1.2} in the space $L^1(\bbG_{N})$ is as follows.
 
 We need the following technical non-degeneracy condition as in \cite{GH}, required for the proof of  the regularity estimate in \eqref{e3.10}. First, in order to have spatial regularity of  kinetic solutions we can localize the $\chi$-function associated to such solution, multiplying it by some $\phi\in C_c^\infty(\R^N)$,  and so, for $\ell\in \N$ sufficiently large,  we may view our localized $\chi$-functions as periodic with periodic
cell $\ell\bbT^N:=[0,\ell]^N$. Since $\ell\Z^N\subset\Z^N$, for any $\ell\in\N$, we   formulate the non-degeneracy condition below in $\Z^N$, as in \cite{GH}.  

The symbol is defined by
$$
\LL(i\tau,i n, \xi):=i(\tau+b(\xi)\cdot n) + {n}^\top \bfa(\xi)n,
$$  
where $b(\xi)=\bff'(\xi)$,   $n\in \Z^N$.  
For $J,\d>0$ and $\eta\in C_b^\infty(\R)$ nonnegative, let 
\begin{equation*}
\begin{aligned}
\Om_{\LL}^\eta(\tau,\eta;\d)&:=\{\xi\in \supp\eta\,:\, |\LL(i\tau, i n, \xi)|\le \d\},\\
\om_{\LL}^\eta(J;\d) &:= \sup_{\tiny\begin{matrix} \tau\in\R, n\in \Z^N\\ |n|\sim J\end{matrix}}|\Om_{\LL}^\eta(\tau,i n;\d)|.
\end{aligned}
\end{equation*}
Let $\LL_{\xi}:=\po_\xi\LL$. We suppose that there exist $\a\in (0,1)$, $\g>0$ and a measurable function  $\vartheta\in L_\loc^\infty(\R;[1,\infty))$ such that
\begin{equation}\label{e3.nondeg1}
\begin{aligned}
\om_{\cL}^\eta(J;\d) &\lesssim_\eta \left(\frac{\d}{J^\g}\right)^\a,\\
\sup_{\tiny{\begin{matrix}\tau\in\R,n\in\Z^N \\ |n|\sim J \end{matrix}}}\sup_{\xi\in\supp \eta}\frac{|\cL_{\xi}(i\tau,in;\xi)|}{\vartheta(\xi)}&\lesssim_\eta J^\g,\qquad \forall \d>0,\, J\gtrsim 1,
\end{aligned}
\end{equation}
where we employ the usual notation $x\lesssim y$, if $x\le Cy$, for some absolute constant $C>0$, and $x\sim y$, if $x\lesssim y$ and $y\lesssim x$. 
Let us point out that, as in \cite{DHV,GH,FLMNZ2}, the symbol $\LL^\ve$ obtained by replacing $\mathbf{a}(\xi)$ by $\mathbf{a}^\ve(\xi):=\mathbf{a}(\xi)+\ve I$, satisfies the non-degeneracy condition \eqref{e3.nondeg1} uniformly in $\ve$.

  \begin{theorem}[Well posedness in $L^1(\bbG_{N})$]~\label{T:1.0}    
  Assume condition \eqref{e3.nondeg1} holds.  Given $T>0$ and  $u_0\in L^1(\bbG_{N})$, there exists a 
  $L^1(\bbG_{N})$-semigroup solution of \eqref{e1.1}-\eqref{e1.2} with  initial data $u_0$, belonging to $L^1(\Om; L^\infty([0,T], L^1(\bbG_N)))$.
  Moreover, let 
 $u_1(t,x), u_2(t,x)$  be two $L^1(\bbG_{N})$-semigroup solutions of \eqref{e1.1}-\eqref{e1.2} with  initial data $u_{01}, u_{02}\in L^1(\bbG_{N})$. 
 Then,  a.s.,  for a.e.\ $t\in[0,T]$, 
 \begin{equation}\label{e5.4AP} 
 \int_{\bbG_N}|u_1(t)- u_2(t)|\,d\mm \le \, C(T)\int_{\bbG_N}|u_{01}-u_{02}|\,d\mm,
 \end{equation}
 for some constant $C(T)$ which depends on the data of the problem and may also depend on 
 $\om\in\Om$. 
 \end{theorem}

 \medskip
 
  Concerning the  existence and uniqueness of invariant measures our main result are as follows. Here we will need to restrict our analysis to almost periodic functions whose spectrum is contained in a finitely generated additive group.  Let $\l_i\in\R^N$, $i=1,\cdots,P$, $\Lambda:=\{\l_1,\cdots,\l_P\}$ be a linearly independent set over $\Z$,  and let $\G_\Lambda$ be the additive subgroup of $\R^N$ generated $\Lambda$.  
 We assume that $\Sp(g_k)\subset \G_\Lambda$, for all $k\in\N$. We also only consider initial data $u_0$ satisfying $\Sp(u_0)\subset \G_\Lambda$. We denote by $\AP_*(\R^N)$, $\B_*^1(\R^N)$and $L^1(\bbG_{*N})$ the subspaces of $\AP(\R^N)$, $\B^1(\R^N)$ and $L^1(\bbG_N)$,
 respectively,  whose functions $g$ satisfy $\Sp(g)\subset\G_\Lambda$.   
 
 In this case we have that the unique 
 $L^1(\bbG_N)$-semigroup solution given by our existence and uniqueness results just described satisfy $\Sp(u(t,x))\subset\G_\Lambda$ and we have, in \eqref{e5.4AP},  $C(T)\equiv1$. In particular, we may define a Markov contraction transition semigroup $P_t$ on the bounded Borel functions on $L^1(\bbG_{*N})$, $\B_b(L^1(\bbG_{*N}))$,  in the usual way ({\em cf.}, e.g., \cite{DV2,CPa,EFM}), namely,  $P_t(\phi)(u_0)=\bbE(\phi(u(t)))$, for $\phi\in\B_b(L^1(\bbG_{*N}))$.

 For $\b\in\G_{\Lambda}$, 
 $\b=n_1\l_1+n_2\l_2+\cdots+ n_P\l_P$, $n_j\in\Z$,  let  ${\bf n}_\b:=(n_1,\cdots,n_P)$ and $|{\bf n}_\b|:=\sqrt{n_1^2+\cdots+n_{P}^2}$. 
   
 Let $b:=\bff'$  and  $\vartheta(\xi)=(1+|\xi|^2)^{-1}$ set
 \begin{equation}\label{e5.NDC1}
 \iota^\vartheta(\d)=\sup_{\a\in\R,\b\in \G_{\Lambda}}
 \int_\R\frac{\d(\abf(\xi):\frac{\b}{|{\bf n}_\b|}\otimes\frac{\b}{|{\bf n}_\b|}+\d)}
 {(\abf(\xi):\frac{\b}{|{\bf n}_\b|}\otimes\frac{\b}{|{\bf n}_\b|}+\d)^2+\d^{\nu}|b(\xi)\cdot \frac{\b}{|{\bf n}_\b|}+\a|^2}\vartheta(\xi)\, d\xi,
 \end{equation}
 for some $\nu\in (1,2)$.
 We assume that 
 \begin{equation}\label{e5.NDC}
 \iota^\vartheta(\d)\le c_1^\vartheta \d^\k,
 \end{equation}
 for some $c_1^\vartheta>0$ and $0<\k<1$. 
 
\begin{theorem}\label{T:1.1}  Assume condition \eqref{e5.NDC} holds, with 
$\iota^\vartheta(\d)$ defined by \eqref{e5.NDC1}.
 Then there is a unique  invariant measure for the transition semigroup $P_t$ in 
 $\B_b(L^1(\bbG_{*N}))$.
 \end{theorem}
 
Assumption \eqref{e5.NDC} is motivated by the non-degeneracy condition imposed in \cite{CPa}  and, except for the presence of the function $\vartheta$, arises naturally  from the latter through the   reduction to the periodic case procedure, described in Section~\ref{S:5}. It plays a crucial role  in connexion with the regularity estimate for periodic solutions proved in \cite{CPa} (see \eqref{e5.DV}), which extends the one for the hyperbolic case established in \cite{DV2}. However, we note that due to the Lipschitz continuity assumptions that we impose on the flux function $\mathbf{f}$ and on the viscosity matrix $\Abf$ (i.e. boundedness of $b(\xi)$ and of $\abf(\xi)$) the integral in \eqref{e5.NDC1} cannot converge without the presence of a weight function $\vartheta$, which is why we introduce it here. Nevertheless,  we can still deduce the referred regularity estimate with minor modifications in the proof contained in \cite{CPa}, as explained in Appendix~\ref{s:A} to which we refer for details.

 \subsection{Earlier works} 
 
 The subject of the asymptotic behaviour of oscillatory solutions of deterministic conservation laws has a very long history that goes back to the first papers on scalar conservation laws (see, e.g., \cite{Ho, Ol, Lx}).  With the introduction of new compactness frameworks such as compensated compactness in, e.g., \cite{Ta, DP1, DP2}, kinetic formulation and averaging lemmas in, e.g., \cite{LPT}, this research gained a great impulse  (see, e.g., \cite{CF0,Fr0,CP, Pv0,Pv, FL, Pv2, GS, GS2}, among others). We also mention the elegant approach in \cite{Da}, using infinite dimensional dynamical systems ideas. The decay of entropy solutions of degenerate parabolic-hyperbolic equations was first established in \cite{CP2}. The latter was extended to Besicovitch almost periodic solutions in \cite{FL} and then also, by a different approach, in \cite{Pv2}.  

On the other hand, in the context of stochastic scalar conservation laws, the study of the asymptotic behaviour of periodic solutions was inaugurated with \cite{EKMS} for the Burgers equation, based on  infinite dimension dynamical systems ideas, which here seems to be the appropriate  approach. The result in \cite{EKMS} was extended to more general conservation laws in several space dimensions in \cite{DV2}. 
The latter was extended to degenerate parabolic-hyperbolic equations in \cite{CPa}. 
We refer to \cite{DV2} and \cite{CPa}  for more references on the subject of invariant measures for stochastic conservation laws  and degenerate parabolic-hyperbolic equations, and other correlated stochastic partial differential equations. Also see \cite{DZ} for a general account on the basic concepts of infinite dimensional dynamical systems associated with stochastic equations.

We finally make a brief comparison between the present paper and  the companion paper \cite{EFM}, which deals with the hyperbolic case.  First, as already mentioned, in the latter,  the contraction property comes almost directly from the Kruzhkov inequality proved in \cite{KJ} 
and holds for any $L^1(\bbG_N)$ solution, while in the present paper we deduce a semi-contraction property using the Kruzhkov-type inequality in the proof of the $L^1$-stability result for the Cauchy problem in \cite{CK}, which then yields a constant $C(T)$ in the quasi-contraction inequality which in general depends on $T$, the data of the problem  and, here, may also depend on  $\om\in\Om$. Second, the non-degeneracy condition in the present case is more involving and also demanded improvements as pointed out in the comment just after the statement of Theorem~\ref{T:1.1}, concerning the assumption \eqref{e5.NDC}. Third, the Proposition~\ref{P:2.1}, establishing the $L^1$-mean semi-contraction inequality, is new  and based on a deterministic formulation motivated by \cite{KJ}, using a Kruzhkov-type inequality obtained in \cite{CK}. Fourth, the well-posedness theory in the present case is much more involving combining ideas of many different sources and original ones, as, for instance, in the proof of Theorem~\ref{T:4.new}, among others.  Finally, the reduction to the periodic case procedure  in the present context of a stochastic degenerate parabolic-hyperbolic equation  is also more complex and has no earlier deterministic equivalent. 
Overall, the present parabolic-hyperbolic case is more difficult and required new ideas in many different points.

\subsection{Plan of the paper} This paper is organized  as  follows. We first establish the well-posedness of the Cauchy problem  \eqref{e1.1}-\eqref{e1.2} of Besicovitch almost periodic ($\BAP$) entropy solutions in Section~\ref{S:2}, Section~\ref{S:3} and Section~\ref{S:4}. Also in Section~\ref{S:4} we establish the existence of entropy solutions in $\bbG_N$ for \eqref{e1.1}-\eqref{e1.2} and introduce the notion of semigroup solutions for which the contraction property in $L^1(\bbG_N)$ holds. In Section~\ref{S:5} we discuss the method of reduction to the periodic case, originally introduced in \cite{Pv}, restricting the well-posedness analysis to $L^1(\bbG_{*N})$, establishing an isometric correspondence between $L^1(\bbG_{*N})$-semigroup solutions and entropy periodic solutions in $L^1(\bbT^P)$, which, in particular, yields the contraction property in $L^1(\bbG_{*N})$. 
Finally, in Section~\ref{S:6}, we establish the existence and uniqueness of an invariant measure for \eqref{e1.1} in $L^1(\bbG_{*N})$.


\section{Almost periodic solutions}\label{S:2}

In this section we define $\BAP$-entropy solutions of \eqref{e1.1}-\eqref{e1.2}, we also establish a semi-contraction property and the stability in a weighed space $L_{\rho_*}^1(\R^N)$.  

Let us define
$$
J(x,t):=\sum_{k=1}^\infty g_k(x) \b_k(t),
$$
and
$$
w=u-J.
$$

Equation \eqref{e1.1} may be written in the form
\begin{equation*}
w_t+\div \fbf(w+J)-D_x^2:\Abf(w+J)=0,
\end{equation*}
or in the form
\begin{equation}\label{e2.1''}
w_t+\div (\fbf(w+J)-\abf(w+J)\nabla_{x}J)-\div_{x}(\abf(w+J)\nabla_{x}w)=0.
\end{equation}   
As for the initial condition we have
\begin{equation}\label{e2.1'''}
w(0,x)=u_0(x),\quad x\in\R^N.
\end{equation}
Equation \eqref{e2.1''} is of the general form
\begin{equation}\label{e2.1iv}
v_t+\sum_{i=1}^N\po_{x_i} b^i(v,t,x)-\sum_{i,j=1}^N\po_{x_i}\left(a^{ij}(v,t,x)\po_{x_j}v\right)=0,
\end{equation}
where $a(v,t,x)=(a^{ij}(v,t,x))_{i,j=1}^N$ is a symmetric nonnegative $N\X N$ matrix. Writing  \eqref{e2.1''} 
in terms of equation \eqref{e2.1iv}, making $b(v,t,x)=(b^1(v,t,x),\cdots,b^N(v,t,x))$, we have
\begin{equation}\label{e2.1new}
\begin{aligned}
& b(w,t,x)=\bff(w+J)-\abf(w+J)\nabla_{x}J,\\
& a(w,t,x)=\abf(w+J).
\end{aligned}  
\end{equation}
In particular, $a(t,x,w)$, as defined in \eqref{e2.1new}, clearly satisfies 
$$
a_{ij}(w,t,x)=\sum_{k=1}^K\s_{ik}(w,t,x)\s_{jk}(w,t,x), \quad K=N,
$$
where $\s_{ik}(w,t,x)=\bar\s_{ik}(w+J)$ and $\bar \s(u)$ is the $N\X K$ such that $\abf(u)=\bar\s(u)\bar \s(u)^\top$, where we write $K$ instead of $N$ in order to preserve the generality; later on we will make use of this generality.  
In this way, we may define a solution for \eqref{e2.1''}-\eqref{e2.1'''} using the definition of solution of the Cauchy problem for the general equation \eqref{e2.1iv} as in \cite{CK}.  Let us denote, as in \cite{CK}, 
\begin{align*}
& \z_{ik}(w,t,x)=\int_0^w\s_{ik}(v,t,x)\,dv,\\
&\z_{ik}^\psi(w,t,x)=\int_0^w\psi(v)\s_{ik}(v,t,x)\, dv,\quad \text{for $\psi\in C(\R)$}.
\end{align*}
Given any convex $C^2$  function $\eta:\R\to\R$, we define the entropy fluxes
$$
q(\cdot,t,x)=(q_i(\cdot,t,x)):\R\to\R^N,\qquad r(\cdot,t,x)=(r_{ij}(\cdot,t,x)):\R\to\R^{N\X N}
$$
by
$$
q_w(w,t,x)=\eta'(w)b_w(w,t,x),\qquad r_w(w,t,x)=\eta'(w)a(w,t,x).
$$
We refer to $\eta$ as an entropy function and $(\eta,q,r)$ as an entropy-entropy flux triple. We are going to consider only the family $\mathcal{E}$ of $C^2$ convex entropies that are  Lipschitz continuous with $\eta''\in C_c(\R)$. Important examples are the $C^2$ approximations of the Kru\v zkov entropies 
$|\cdot-c|$, $c\in\R$. More specifically,  we can consider a $C^1$ approximation of the function 
$\sgn(\cdot)$, such as ({\em cf.} \cite{CK}), for $\ve>0$,
\begin{equation}\label{e2.sgn}
\sgn_\ve(\xi)= \begin{cases} -1, & \xi<-\ve,\\ \sin(\frac{\pi}{2\ve}\xi), & |\xi|\le \ve\\
1, & \xi>\ve.
\end{cases}
\end{equation}
For $c\in\R$, we then get the convex entropy function in the family $\mathcal{E}$
$$
u\mapsto \eta_\ve(w,c)=\int_c^w\sgn_\ve(\xi-c)\,d\xi. 
$$

\begin{definition}\label{D:2.1} Let $u_0\in L^\infty(\R^N)\cap \B^1(\R^N)$ and  $T>0$ be given. A $L^1_\loc(\R^N)\cap \B^1(\R^N)$-valued stochastic process,  adapted
to $\{\F_t\}$, is said to be a $\BAP$-entropy solution of \eqref{e1.1}-\eqref{e1.2} if, for almost all $\om\in\Om$, for $w=u-J$,
\begin{enumerate}
 \item[(i)] $w(t)$ is $L_\loc^1(\R^N)\cap \B^1(\R^N)$-weakly continuous on $[0,T]$.
 \item[(ii)] $w\in L^\infty([0,T]; L_\loc^1(\R^N))\cap \B^1(\R^N)$.  
 \item[(iii)] (Weak regularity) For $k=1,\cdots,K$,
 $$
 \sum_{i=1}^N\left(\po_{x_i}\z_{ik}(w,t,x)-\z_{ik,x_i}(w,t,x)\right)\in L_\loc^2((0,T)\X\R^N).
 $$
 \item[(iv)] (Chain Rule) For $k=1,\cdots,K$,
 $$
 \sum_{i=1}^N\left(\po_{x_i}\z_{ik}^\psi(w,t,x)-\z_{ik,x_i}^\psi(w,t,x)\right)=\psi(w)\sum_{i=1}^N\left(\po_{x_i}\z_{ik}(w,t,x)-\z_{ik,x_i}(w,t,x)\right)
 $$
 a.e.\ in $(0,T)\X\R^N$, for any $\psi\in C_b(\R)$.
 
 \item[(v)] (Entropy Inequality) For any entropy-entropy flux triple $(\eta,q,r)$, with $\eta\in\mathcal{E}$,
 \begin{multline}\label{e2.entropy}
 \po_t\eta(w)+\sum_{i=1}^N\po_{x_i}q_i(w,t,x)-\sum_{i,j=1}^N\po_{x_ix_j}^2r_{ij}(w,t,x)\\
 +\sum_{i=1}^N\left(\eta'(w)b_{i,x_i}(w,t,x)-q_{i,x_i}(w,t,x)\right)+\sum_{i,j=1}^N\po_{x_i}r_{ij,x_i}(w,t,x)\\
 \le -\eta''(w)\sum_{k=1}^K\left(\sum_{i=1}^N\left(\po_{x_i}\z_{ik}(w,t,x)-\z_{ik,x_i}(w,t,x)\right)\right)^2\quad \text{a.s. in $\DD'((0,T)\X\R^N)$}.
 \end{multline}
 
 \item[(vi)] (Initial Condition) For any $R>0$, 
\begin{equation}\label{e2.D21'}
\lim_{t\to0+}\intl_{C_R}|u(t,x)-u_0(x)|\,dx=0,
\end{equation}
 \end{enumerate}
 \end{definition}
  
  The following result was proved in \cite{CK} (see equation (4.12) therein). We slightly modified the formula by including a weight function $\rho_*(x)>0$, $\rho_*\in C^2(\R^N)\cap L^\infty(\R^N)$, with $\sum_{i=1}^N|\po_{x_j}\rho_*(x)|+\sum_{i,j=1}^N|\po_{x_i x_j}^2\rho_*(x)|\le C_0\rho_*(x)$,
  for all $x\in\R^N$, e.g., $\rho_*(x)=e^{-\sqrt{1+|x|^2}}$, which follows immediately by the computations in \cite{CK} and we refer to \cite{CK} for the proof.
  
 \begin{proposition}\label{P:New} Given two $\BAP$-entropy solutions $u(t,x),\ v(t,x)$, setting $w=u-J$, $\hat w=v-J$,
 for a.a.\ $\om\in\Om$, in the sense of the distributions on $(0,T)\X\R^N$,  for some constant $C>0$ depending on $T$, the data of the problem, and possibly on $\om$, we have 
 \begin{multline}\label{e2.30}
\int\limits_{\R^N\X(0,T)}|w-\hat w|\varphi_t\,\rho_*(x) dx\,dt\\
+\int\limits_{\R^N\X(0,T)} \sgn(w-\hat w)\left((\bff(w+J)+\abf(w+J)\nabla J)\right.\\
-\left.(\bff(\hat w+J)+\abf(\hat w+J)\nabla J)\right)\cdot\nabla\varphi \,\rho_*(x)dx\,dt\\
+\int\limits_{\R^N\X(0,T)} \sgn(w-\hat w)(\Abf(w+J)-\Abf(\hat w+J)):\nabla^2\varphi \,\rho_*(x) dx\,dt\\
\ge -C(\max_{t,x}|\varphi|+\max_{t,x,i}|\po_{x_i}\varphi|)\int\limits_{((0,T)\X\R^N)\cap\supp(\varphi)}|w-\hat w|\,\rho_*(x) dx\,dt.
\end{multline}
 \end{proposition}
 
 As a consequence of \eqref{e2.30} we have both the stability in $L_{\rho_*}^1(\R^N)$ and the stability in $\B^1(\R^N)$. 
 First we establish  the stability in $\B^1$; its statement and proof are motivated by proposition~1.3 in \cite{Pv}.  Later on we will see that the constant $C>0$ may be taken equal to 1, at least for solutions in some separable subspaces of $\B^1$.

\begin{proposition}[$L^1$-mean semi-contraction property]~\label{P:2.1} Let 
 $ u(t,x), v(t,x)$  be two $\BAP$-entropy solutions of \eqref{e1.1}-\eqref{e1.2} with initial data $u_0(x), v_0(x)$. 
 Then,  a.s., for some $C>0$ depending on $t$, on the data of the problem, and possibly on $\om$,  for a.e.\ $t>0$, 
 \begin{equation}\label{e2.4AP} 
 N_1(u(t,\cdot)- v(t,\cdot)) \le C\, N_1(u_0-v_0).
 \end{equation}
 \end{proposition}
 
 \begin{proof}   We apply \eqref{e2.30} with $\rho_*\equiv1$.   We define a sequence approximating the indicator function of the interval $(t_0,t_1]$ , by setting for $\nu\in\N$,
$$
\d_\nu(s)=\nu\l(\nu s),\quad \theta_\nu(t)=\int_0^t\d_\nu(s)\,ds=\int_0^{\nu t}\l(s)\,ds,
$$
where $\l\in C_c^\infty(\R)$, $\supp \l\subset [0,1]$, $\l\ge0$, $\int_{\R}\l(s)\,ds=1$.   We see that $\d_\nu(s)$ converges to the Dirac measure in the sense of distributions in 
$\R$ while $\theta_\nu(t)$ converges everywhere to the Heaviside function. For $t_1>t_0>0$, if $\chi_\nu(t)=\theta_\nu(t-t_0)-\theta_\nu(t-t_1)$, then $\chi_\nu\in C_c^\infty(\R_+)$,
$0\le\chi_\nu\le 1$, and the sequence $\chi_\nu(t)$ converges everywhere, as $\nu\to\infty$, to the indicator function of the interval $(t_0,t_1]$. Let us take $g\in C_c^\infty(\R^d)$, satisfying $0\le g\le 1$, $g(y)\equiv 1$ in the cube $C_1$, $g(y)\equiv 0$ outside the cube $C_k$, with $k>1$. We apply \eqref{e2.30} to the test function $\varphi=R^{-N}\chi_\nu(t)g(x/R)$, for $R>0$. We then get
\begin{multline}\label{e2.4}
\int_0^\infty\bigl(R^{-N}\intl_{\R^N}|u(t,x)-v(t,x)|g(x/R)\,dx\bigr)(\d_\nu(t-t_0)-\d_\nu(t-t_1))\,dt\\
+R^{-N-1}\iint_{\R_+^{N+1}}\sgn(w-\hat w)((\bff(w+J)+\abf(w+J)\nabla J)\\
-(\bff(\hat w+J)+\abf(\hat w+J)\nabla J))\cdot\nabla_yg(x/R)\chi_\nu(t)\,dx\,dt\\
+R^{-N-2}\sum_{i,j=1}^\k\iint_{R_+^{N+1}}\sgn(w-\hat w)( A_{ij}(w+J)-A_{ij}(\hat w+J))\po_{y_iy_j}^2g(x/R)\chi_\nu(t)\,dx\,dt\\
\ge-C(\frac1{R^N}+\frac1{R^{N+1}})\int_{(t_0,t_1)\X  C_{kR}}|u-v|\,dx\,dt.
\end{multline}
Define
$$
F=\{t>0\,:\, \text{$(t,x)$ is a Lebesgue point of $|u(t,x)-v(t,x)|$ for a.e.\ $x\in\R^d$} \}.
$$
As a consequence of Fubini's theorem, $F$ is a set of full Lebesgue measure and so each $t\in F$ is a Lebesgue point of the functions
$$
I_R(t)=R^{-N}\int_{\R^N}|u(t,x)-v(t,x)| g(x/R)\,dx,
$$
for all $R>0$ and all $g\in C_c(\R)$ (see, e.g., lemma~1.3 in \cite{Pv}). Now we assume $t_0,t_1\in F$ and take the limit as $\nu\to\infty$ in \eqref{e2.4} to get
\begin{multline}\label{e2.5}
I_R(t_1)\le I_R(t_0)+ R^{-N-1}\iintl_{(t_0,t_1)\X\R^N}\sgn(w-\hat w)((\bff(w+J)+\abf(w+J)\nabla J)\\
-(\bff(\hat w+J)+\abf(\hat w+J)\nabla J))\cdot\nabla_y g(x/R)\,dx\,dt\\
+R^{-N-2}\sum_{i,j=1}^N\iintl_{(t_0,t_1)\X\R^N}\sgn(w-\hat w) ( A_{ij}(w+J)-A_{ij}(\hat w+J))\po_{y_iy_j}^2 g(x/R)\,dx\,dt\\
+C(\frac1{R^N}+\frac1{R^{N+1}})\int_{(t_0,t_1)\X  C_{kR}}|u-v|\,dx\,dt.
\end{multline}
Now, we have
\begin{multline}\label{e2.6}
R^{-N-1}\biggl|\iintl_{(t_0,t_1)\X\R^N}\sgn(w-\hat w)((\bff(w+J)+\abf(w+J)\nabla J)\\
-(\bff(\hat w+J)+\abf(\hat w+J)\nabla J))\cdot\nabla_y g(x/R)\,dx\,dt\biggr|\\
\le k^{N}R^{-1}(\Lip(\bff)+\|\nabla J\|_\infty\Lip(\abf)) \|\nabla g\|_\infty  (kR)^{-N}\iintl_{(t_0,t_1)\X C_{kR}}|u(t,x)- v(t,x)|\,dx\,dt  \\
\longrightarrow 0,
\quad \text{as $R\to\infty$},
\end{multline}
which follows since a.s.\ $u(t,x), v(t,x)\in L^\infty([0,T]; \B^1(\R^N))$. 
Also, we have
\begin{multline}\label{e2.7}
R^{-N-2}\left\vert \sum_{i,j=1}^\k\iint_{R_+^{N+1}}\sgn(w-\hat w) ( A_{ij}(w+J)-A_{ij}(\hat w+J))\po_{y_iy_j}^2g(x/R)\chi_\nu(t)\,dx\,dt\right\vert \\ 
\le k^N R^{-2}\Lip(\Abf) \|\nabla^2 g\|_\infty (kR)^{-N}\iintl_{(t_0,t_1)\X C_{kR}}|w(t,x)-\hat w(t,x)|\,dx\,dt\to 0,\\
  \text{as $R\to\infty$}
\end{multline}
which also follows because a.s.\ $u(t,x), v(t,x)\in L^\infty([0,T];\B^1(\R^N))$.

 On the other hand, we have
 $$
 N_1(u(t,\cdot)-v(t,\cdot))\le \limsup_{R\to\infty} I_R(t)\le k^N N_1(u(t,\cdot)-v(t,\cdot)),
 $$
 so taking the limit as $R\to\infty$ in \eqref{e2.5} and applying Gr\"onwall,  for $t_0, t_1\in F$, $t_0<t_1$,   we get 
 $$
 N_1(u(t_1,\cdot)-v(t_1,\cdot))\le k^{N}C N_1(u(t_0,\cdot)-v(t_0,\cdot)),
 $$
 and since $k>1$ is arbitrary we can make $k\to 1+$ to get the desired result. Finally, for $t_0=0$, we use \eqref{e2.D21'} to send $t_0\to0+$ in \eqref{e2.5} to obtain \eqref{e2.4AP}. 
  
\end{proof}

\begin{remark}\label{R:2.1} We remark that for functions $g\in\B^1(\R^N)$ the $N_1(g)$ coincides with the mean value of $|g|$ and so with the 
norm in $L^1(\bbG_N)$,  by the isometric isomorphism between $\B^1(\R^N)$ and $L^1(\bbG_N)$ with the Haar measure $\mm$ induced by the mean-value.  
Therefore, we may write \eqref{e2.4AP} as
\begin{equation}\label{e2.4AP'}
\int_{\bbG_N}|u(t)-v(t)|\,d\mm\le C\int_{\bbG_N}|u_0-v_0|\,d\mm.
\end{equation} 
\end{remark}
 
Also, form \eqref{e2.30}, it follows immediately the $L_{\rho_*}^1$-stability, with $\rho_*$ decaying sufficiently fast for $|x|\to\infty$, such as $\rho_*(x)=e^{-\sqrt{1+|x|^2}}$. 
 
\begin{theorem}[{\em cf.} \cite{CK}]\label{T:2.1} Let 
 $ u(t,x), v(t,x)$  be two solutions of \eqref{e1.1}-\eqref{e1.2} with initial data $u_0(x), v_0(x)\in L^\infty\cap\B^1(\R^N)$. Let 
 $\rho_*(x)=e^{-\sqrt{1+|x|^2}}$.
 Then,  a.s., for some $C>0$ depending on $t$, on the data of  the problem, 
 and possibly on $\om$, for a.e.\ $t>0$,
 \begin{equation}\label{e2.stab}
 \int_{\R^N}|u(t,x)-v(t,x)|\,\rho_*(x) dx\le C\int_{\R^N}|u_0(x)-v_0(x)|\,\rho_*(x) dx.
 \end{equation} 
\end{theorem}

\section{Approximate solutions}\label{S:3}

\subsection{First approximation}\label{SS:3.1}

In this and the next subsection we assume  $u_0\in \AP(\R^N)$. We consider  first the fourth order approximation for \eqref{e1.1}-\eqref{e1.2}, with $\Abf^\ve(u)=(\Abf+\ve I)(u)$, 
\begin{equation}\label{e3.1}
du+\div(\bff(u))\,dt=D^2:\Abf^\ve(u)\,dt-\mu\Delta^2u\,dt +\Phi_\ve \,dW(t), 
\end{equation}
$x\in\R^N$,  $t\in(0,T)$, with initial condition
\begin{equation}\label{e3.2}
u(0,x)=u_{0\ve}(x),\quad x\in\R^N,
\end{equation}
where $u_{0\ve}$ is a trigonometric polynomial approximating $u_0$ in $\AP(\R^N)$
and $\Phi_\ve e_k=g_{k\ve}$ where $g_{k\ve}\equiv 0$ for $k>1/\ve$ and $g_{k\ve}$ is a trigonometric polynomial approximating $g_k$ uniformly in $\AP(\R^N)$ for $k\le 1/\ve$, $k\in\N$. 

Let $J_\ve=\int_0^t \Phi_\ve dW$. For $w=u-J_\ve$ we write \eqref{e3.1} as
\begin{equation}\label{e3.3}
w_t+\div \fbf (w+J_\ve)-D^2:\Abf^\ve(w+J_\ve)= -\mu \Delta^2 w -\mu \Delta^2 J_\ve.
\end{equation} 

For $g\in \mathcal{S}'(\R^N)$, the space of Schwartz distributions in $\R^N$, let us denote by $\hat g$ or $\mathcal{F}(g)$ the Fourier transform of $g$ and by $\check g$ or $\mathcal{F}^{-1}(g)$ the inverse Fourier transform. 
Also, let us set 
$$
\Phi_*(x)= C_* \mathcal{F}^{-1}(e^{-|y|^4})(x),
$$
where $C_*$ is chosen so that $\int_{\R^N}\Phi_*(x)\,dx=1$. Let us denote
$$
K(t,x)=\frac{1}{t^{n/4}}\Phi_*(\frac{x}{t^{1/4}}).
$$
We can verify that $K(t,x)$ is the fundamental solution of the equation
$$
u_t=-\Delta^2 u,
$$
and $K_\mu(t,x)=K(\mu t, x)$ is the fundamental solution of
$$
u_t=-\mu \Delta^2 u.
$$
The most important facts about $K_\mu$ for us are the following 
$$
\|K_\mu(t)\|_1=1,\quad \|D^k K_\mu(t)\|_1\le \frac{C_k}{(\mu t)^{k/4}},\quad k\in\N,
$$
where the constants $C_k$ only depend on $k$. 
Since we are assuming $\bff\in C^3(\R; \R^N)\cap \Lip(\R; \R^N)$ and $\Abf\in C^3(\R;\R^{N\X N})\cap \Lip(\R;\R^{N\X N})$, we can obtain a solution to \eqref{e3.1}, \eqref{e3.2}, with 
$w, \nabla w, \nabla^2 w, \nabla^3 w, \nabla^4 w, w_t\in C([0,T];L^\infty(\R^N))$ in a standard
way beginning  by a fixed point argument for the functional
\begin{multline}\label{e3.4}
\LL(w)(t)=K_\mu(t)*u_{0\ve}-\sum_{i=1}^N\int_0^t\po_{x_i} K_\mu(t-s)*f_i(w+J_\ve)(s)\,ds\\
+\sum_{i,j=1}^N\int_0^t \po_{x_i x_j}^2K_\mu(t-s)* A_{ij}^\ve(w+J_\ve)(s)\,ds\\
- \int_0^tK_\mu(t-s)*\mu\Delta^2 J_\ve(s)\,ds.
\end{multline}
After proving the existence of a smooth solution to 
\eqref{e3.1}-\eqref{e3.2}, we proceed as in \cite{DHV}: first we consider the limit as $\mu\to0$, with $\ve>0$ fixed; then we consider the limit as $\ve\to0$.  

We have the following analogue of proposition~4.3 of \cite{DHV}. 

\begin{proposition}\label{P:3.1} Let $u^{\mu,\ve}$ be the solution of \eqref{e3.1}-\eqref{e3.2}. For $p\ge2$, $\rho_*$ as in Theorem~\ref{T:2.1}, the following estimate holds
\begin{multline}\label{eP31}
\bbE \sup_{0\le t\le T} \|u^{\mu,\ve}(t)\|_{L_{\rho_*}^2(\R^N)}^p+p\ve\bbE\int_0^T\|u^{\mu,\ve}\|_{L_{\rho_*}^2(\R^N)}^{p-2}\|\nabla u^{\mu,\ve}\|_{L_{\rho_*}^2(\R^N)}^2\,ds\\
\le C(1+\bbE\|u_0\|_{L_{\rho_*}^2(\R^N)}),
\end{multline}
where $L_{\rho_*}^2(\R^N)$ is the $L^2$ space with respect to the measure $\rho_*(x)\,dx$, and the constant $C$ does not depend on $\mu,\ve$. 
\end{proposition}

We also have the following proposition concerning the $\B^2$-norm of $u^{\mu,\ve}$, which also follows from It\^o's formula. 

\begin{proposition}\label{P:3.new} Let $u^{\mu,\ve}$ be the solution of \eqref{e3.1}-\eqref{e3.2}. Then, 
$u^{\mu,\ve}(t)\in\B^2(\R^N)$, for all $t\in[0,T]$, and we have
 \begin{equation}\label{e3.12}
\bbE\sup_{0\le t\le T}\|u^{\mu,\ve}(t)\|_{\B^2}^2\le \|u_0\|_{\B^2}^2+D_0T.
\end{equation}
\end{proposition}

\begin{proof}  Since we are assuming that $u_{0\ve}\in\AP(\R^N)$, and the $g_{k\ve}$'s are trigonometrical 
polynomial,  we can prove by induction, using the Duhamel formula  \eqref{e3.4}, that $u^{\mu,\ve}(t)\in\AP(\R^N)$ for all $t\in[0,T]$, first in a small interval $[0,t_0]$, where $\LL$ is a contraction on  $C([0,t_0];L^\infty(\R^N))$, and then, successively in intervals $[k t_0,(k+1) t_0]$, until covering the interval 
$[0,T]$. In particular, $u^{\mu,\ve}(t)\in\B^2(\R^N)$, for all $t\in[0,T]$. We can then use It\^o's chain rule applied to $\Me(u^2)$ to deduce \eqref{e3.12}, where $D_0$ is the same as in \eqref{e1.3}. 

\end{proof}

  For $k\in\Z$,  $H_{\rho_*}^k(\R^N)$ denotes the space of distributions $\ell\in\DD'(\R^N)$ such that  $\rho_*^{1/2} \ell\in H^k(\R^N)$. Observe that, if $g\in L_{\rho_*}^2(\R^N)$, then, for a multi-index $\a$, with 
  $|\a|=j$, $j\in\N\cup\{0\}$,   $D^\a g\in H_{\rho_*}^{-j}(\R^N)$, as it is easy to check by induction in $|\a|$. Indeed, if $g\in L_{\rho_*}^2(\R^N)$, clearly $\rho_*^{1/2}g\in L^2(\R^N)$. If $|\a|=1$, $\rho_*^{1/2} D^\a g=D^\a(\rho_*^{1/2} g)-g D^\a\rho_*^{1/2}$, and, from the properties of $\rho_*$, both $\rho_*^{1/2} g$ and $g D^\a \rho_*^{1/2}$ belong to $L^2(\R^N)$, so $\rho_*^{1/2} D^\a g\in H^{-1}(\R^N)$. Similarly, if the assertion is true for $|\b|=j$, if $|\a|=j+1$, then $D^\a g=D^\g(D^\b g)$, with $|\g|=1$ and $|\b|=j$.
  Again $\rho_*^{1/2} D^\a g=D^\g (\rho_*^{1/2} D^\b g) -(D^\g\rho_*^{1/2})D^\b g$, and since the assertion holds for $|\b|=j$, then both 
  $\rho_*^{1/2} D^\b g$ and  $ (D^\g\rho_*^{1/2})D^\b g=g_\g \rho_*^{1/2}D^\b g$, for some $g_\g\in C_b^\infty(\R^N)$, belong to $H^{-j}(\R^N)$, and so $\rho_*^{1/2} D^\a g\in H^{-(j+1)}(\R^N)$.

   For use in the next subsection, for $k\in\Z$,  let $H_{\loc}^k(\R^N)$ denotes, as usual, the space of distributions $\ell\in\DD'(\R^N)$ such that, for any $\phi\in C_c^\infty(\R^N)$,   $\phi \ell\in H^k(\R^N)$. Again we observe  that, if $g\in L_{\rho_*}^2(\R^N)$, then, for a multi-index $\a$, with 
  $|\a|=j$, $j\in\N\cup\{0\}$,   $D^\a g\in H_\loc^{-j}(\R^N)$, as it is also easy to check by induction in $|\a|$. Indeed, if $g\in L_{\rho_*}^2(\R^N)$, given $\phi\in C_c^\infty(\R^N)$, there exists $C_\phi>0$, such that $|\phi||g|\le C_\phi \rho_*^{1/2} |g|$, which implies that $\phi g\in L^2(\R^N)$. Also, if $|\a|=1$, $\phi D^\a g=D^\a(\phi g)-g D^\a\phi$, and, from what we have just seen, both $\phi g$ and $g D^\a \phi$ belong to $L^2(\R^N)$, so $\phi D^\a g\in H^{-1}(\R^N)$. We then complete the induction proof of the assertion exactly as you just did for the spaces $H_{\rho_*}^{-j}(\R^N)$.

  Again following \cite{DHV}, we have the following proposition corresponding to proposition~4.4 of \cite{DHV}.
 
 \begin{proposition}\label{P:3.2} For all $\l\in(0,1/2)$,  there exists a constant $C>0$, independent  of $\mu$,  such that, 
 for all $\mu\in(0,1)$, 
 $$
 \bbE\| u^{\mu,\ve}\|_{C^\l([0,T]; H_{\rho_*}^{-3}(\R^N))}\le C.
 $$
 \end{proposition}
 
 At this point we note that all the arguments in subsections 4.1 to 4.3 of \cite{DHV} can be repeated line by line, only replacing 
 $L^p(\bbT^N)$ by $L_{\rho_*}^p(\R^N)$ and $H^k(\bbT^N)$ by $H_{\rho_*}^k(\R^N)$, $k\in\Z$,
 in order to get a solution of the non-degenerate parabolic equation obtained as limit when $\mu\to0$, for $\ve>0$ fixed. We omit the details.
 
\subsection{Second approximation}\label{SS:3.2}
 
We now consider the non-degenerate parabolic problem
\begin{equation}\label{e4.1}
du+\div(\bff(u))\,dt=D^2:\Abf^\ve(u)\,dt +\Phi_\ve \,dW(t), 
\end{equation}
$x\in\R^N$,  $t\in(0,T)$, with initial condition
\begin{equation}\label{e4.2}
u(0,x)=u_{0\ve}(x),\quad x\in\R^N,
\end{equation}
where $u_{0\ve}$ is smooth $u_{0\ve}\in \AP(\R^N)$, and $u_{0\ve}\to u_0$ in $\AP(\R^N)$. 
$\Phi_\ve e_k=g_{k\ve}$ where $g_{k\ve}\equiv 0$ for $k>1/\ve$ and $g_{k\ve}$ is a trigonometric polynomial approximating $g_k$ uniformly in $\AP(\R^N)$ for $k\le 1/\ve$, $k\in\N$. Our goal in this subsection is to study the limit when $\ve\to0$. As before, setting $w^\ve=u^\ve-J_\ve$ we write \eqref{e4.1} as
\begin{equation}\label{e4.1new}
w_t+\div \fbf (w+J_\ve)-D^2:\Abf^\ve(w+J_\ve)= 0.
\end{equation} 
 
If $\chi_u(\xi)=1_{\xi< u}(\xi)-1_{\xi<0}$, for $S\in C ^2(\R)$, we have
$$
S(u)-S(0)=\int_{\R}S'(\xi)\chi_u(\xi)\,d\xi.
$$
Using this fact and It\^o's formula applied to $S(u)$ we deduce that $\ff(t,x,\xi)=\chi_{u^\ve(t,x)}(\xi)$ satisfies, recalling that $\bfa(\xi)=\Abf'(\xi)$ and  $\Abf^\ve(\xi)=\Abf(\xi)+\ve \xi I$, 
\begin{multline}\label{e3.9}
\ff_t+b(\xi)\cdot\nabla_x\ff -(\bfa(\xi)+\ve I):D_x^2 \ff= \Big((\ve|\nabla_x u^\ve|^2(\xi)+|\sigma \nabla u^\ve|^2)\d_{\xi=u^\ve}- \frac12 G_\ve^2\d_{\xi=u^\ve}(\xi)\Big)_\xi \\
-\d_{\xi=u^\ve} \Phi\,dW,
\end{multline}
where $b(\xi)=\bff'(\xi)$.

By means of the imposition of a non-degeneracy condition as in \cite{GH},  we can then obtain a regularity estimate from the stochastic averaging lemma by Gess and Hofmanov\'a in \cite{GH} of the type
\begin{equation}\label{e3.10}
\bbE\|\phi u^\ve\|_{L^r(\Om\X[0,T]; W^{s,r}(\R^N))}\le C_\phi(\bbE \|\phi u_0\|^3_{L^3(\R^N)}+1),
\end{equation}
for each $\phi\in C_c^\infty(\R^N)$, for some $C_\phi>0$ depending of $\phi$ but independent of $\ve>0$.   In particular, given an open bounded set, with smooth boundary, $\mathcal{O}$, there exists a constant $C_{\mathcal{O}}>0$, independent of $\ve>0$, such that
 \begin{equation}\label{e3.10'}
\bbE\| u^\ve\|_{L^r(\Om\X[0,T]; W^{s,r}(\mathcal{O}))}\le C_{\mathcal{O}},
\end{equation}

Later on, for the study of the asymptotic behavior, we will state another condition as in \cite{CP} which is actually weaker than \eqref{e3.nondeg1}.

\section{Existence  and uniqueness of $\BAP$-entropy solutions}\label{S:4}

The purpose of this section is to prove the following theorem.  

\begin{theorem}~\label{T:4.1}  Given $T>0$, there is a $\BAP$-entropy solution of \eqref{e1.1}-\eqref{e1.2}. Furthermore, the solution is pathwise unique.
\end{theorem}  

The pathwise uniqueness of the solution of \eqref{e1.1}-\eqref{e1.2} is established in Theorem~\ref{T:2.1}, since we are assuming that the initial data are in $\AP(\R^N)$ and so also in 
$L_{\rho_*}^1(\R^N)$.

As to the existence, we apply a reasoning similar as the one in \cite{FLMNZ}, which follows the method in \cite{DHV} (see also \cite{Ha}). Namely: (i) to apply Kolmogorov's continuity lemma; (ii)  to prove of the tightness of the laws which, by Prokhorov's theorem, implies the compactness of the laws in the weak topology of measures; (iii) to apply Skorokhod's representation theorem; (iv) to show that the limit a.e.\  given by Skorokhod's representation theorem is a martingale entropy solution;  (iv) to apply the Gyongy-Krylov criterion for convergence in probability, using the uniqueness of the solution of \eqref{e1.1}-\eqref{e1.2}, therefore  obtaining the convergence in $L_\loc^1$ of the solutions of \eqref{e4.1}-\eqref{e4.2} to a $L_\loc^1$ function which is an entropy solution of \eqref{e1.1}-\eqref{e1.2}.  We now line up the main results that follow the just described streamline. To begin with, through the application of Kolmogorov's continuity lemma, we have the following analogue of Proposition~\ref{P:3.2}.

\begin{proposition}\label{P:3.2'} Let $u^\ve$ be the solution of \eqref{e4.1}-\eqref{e4.2}. For all $\l\in(0,1/2)$, for each 
$\phi\in C_c^\infty(\R^N)$, there exists a constant $C_\phi>0$, depending on $\phi$ but  independent of $\ve>0$, such that 
 for all $\ve\in(0,1)$
 $$
 \bbE\|\phi u^\ve\|_{C^\l([0,T]; H^{-2}(\R^N))}\le C_\phi.
 $$
 \end{proposition}

In particular, for a fixed $1<r<2$,  for all $\ve\in(0,1)$, for all open bounded with smooth boundary 
$\mathcal{O}$ ,we have 
\begin{equation}\label{e3.100}
 \bbE\|u^\ve\|_{C^\l([0,T]; W^{-2,r}(\mathcal{O}))}\le C_{\mathcal{O}},
 \end{equation}
 for some constant $C_{\mathcal{O}}>0$ independent of $\ve>0$.
 
For a fixed $1<r<2$ such that the regularity estimate \eqref{e3.10}, from \cite{GH}, holds,  let us denote, 
$$
\mathcal{X}=L^r(0,T; L_\loc^r(\R^N))\cap C([0,T]; W_\loc^{-2,r}(\R^N)).
$$
First, we recall that $u\in \mathcal{X}$ if, for each $\phi\in C_c^\infty(\R^N)$, $\phi u\in L^r(0,T; L^r(\R^N))\cap C([0,T];W^{-2,r}(\R^N))$.  Also, convergence of a sequence $u_n\to u$ in 
$\mathcal{X}$ means that for all $\phi\in C_c^\infty(\R^N)$, $\phi u_n\to \phi u$ in  
$L^r(0,T; L^r(\R^N))\cap C([0,T];W^{-2,r}(\R^N))$. Observe that we can endow $\mathcal{X}$ with a metric with respect to which it becomes a separable metric  space. Indeed, for $\nu\in\N$, let 
$\mathcal{O}_\nu$ be the open ball  of radius $\nu$ around the origin in $\R^N$, and let $\phi_\nu\in C_c^\infty(\R^N)$, $0\le \phi_\nu\le 1$, with $\phi_\nu\equiv1$ on $\mathcal{O}_\nu$ and $\phi_\nu\equiv 0$, outside $\mathcal{O}_{\nu+1}$. 
 For $u\in\mathcal{X}$, let $\rho_\nu(u)$ be the norm of $\phi_\nu u$ in $L^r(0,T; L^r(\R^N))\cap C([0,T];W^{-2,r}(\R^N))$. We can then define the following metric in $\mathcal{X}$,
$$
\d(u,v)=\sum_{\nu=1}^\infty 2^{-\nu} \frac{\rho_\nu(u-v)}{1+\rho_\nu(u-v)}, \quad u,v \in\mathcal{X}.
$$

 Concerning the tightness of the laws $\mu_{u^\ve}$, $\ve\in(0,1)$,  associated to the solutions of \eqref{e4.1}-\eqref{e4.2}, $0<\ve<1$, we have the following result (see, e.g., the proof of proposition~5.3 in \cite{FLMNZ}). 
   
\begin{proposition}\label{P:3.2''} Let  $\mu_{u^\ve}$ be the law defined in $\mathcal{X}$ associated with $u^\ve$. The set 
$\{\mu_{u^\ve}\,:\, \ve\in(0,1)\}$ is tight and, therefore, relatively weakly compact in $\mathcal{X}$.

\end{proposition}

\begin{proof} Let us define
\begin{multline*}
K_R=\{ u\in\mathcal{X}\,:\, 
 \|u\|_{L^\infty(0,T; L_{\rho_*}^2(\R^N))}\le R,\\ 
 \|\phi_\nu u\|_{C^\l([0,T];W^{-2,r}(\R^N))}\le 2^\nu(C_{\mathcal{O}_{\nu+1}} +1)R, 
 \\  \|\phi_\nu u\|_{L^r(0,T;W^{s,r}(\R^N))} \leq 2^\nu (C_{\mathcal{O}_{\nu+1}} + 1) R,
 \quad \forall \nu \geq 1\},
 \end{multline*}
where $C_{\mathcal{O}_{\nu+1}}$ is greater than or equal to the constant $C_{\mathcal{O}}$ in
\eqref{e3.10'} and \eqref{e3.100} for $\mathcal{O}=\mathcal{O}_{\nu+1}$, and $\phi_\nu$ is as above. We claim that $K_R$ is a relatively compact subset of $\mathcal{X}$. Indeed, if $\psi^k$ is a sequence in $K_R$, then $\|\psi^k\|_{L^\infty(0,T; L_{\rho_*}^2(\R^N))}\le R$, and
for all $\nu\in\N$, we have that $\phi_\nu \psi^k$ is bounded in 
$C^\l([0,T];W^{-2,r}(\mathcal{O}_{\nu+1}))\cap L^r(0,T;W^{s,r}(\mathcal{O}_{\nu+1}))$. 
Since $\|\psi^k\|_{L^\infty(0,T; L_{\rho_*}^2(\R^N))}\le R$, we can extract a subsequence, still denoted $\psi^k$,
and a $\psi\in L^\infty(0,T; L_{\rho_*}^2(\R^N))$ such that $\psi^k\wto \psi$ in the weak*-weak topology of $L^\infty(0,T; L_{\rho_*}^2(\R^N))$.
For all $\nu\in\N$, for a.e.\ $t\in[0,T]$, we have that $\|\phi_\nu \psi^k(t)\|_{L^2(\R^N)}\le C_{\phi_\nu}R$, for all $k\in\N$. In particular, we can find a dense set in $[0,T]$ and a subsequence, still denoted $\psi^k$, such that $\phi_\nu\psi_k(t)\to \phi_\nu\psi(t)$ in $W^{-1,2}(\R^N)$ strongly , for a dense set of $t\in[0,T]$, for all 
$\nu\in\N$, and so, also in  $W^{-2,r}(\R^N)$.  Since $\phi_\nu\psi^k$ is bounded in $C^\l([0,T];W^{-2,r}(\mathcal{O}_{\nu+1}))$, we deduce that  $\phi_\nu\psi^k\to\phi_\nu \psi$ strongly in $C([0,T];W^{-2,r}(\R^N))$, for all $\nu\in\N$.  
 On the other hand,  by interpolation we have, for all $\varphi\in C_c^\infty(\mathcal{O}_{\nu+1})$,  
$$
\|\varphi\|_{L^r(0,T;W^{2,r}(\cO_{\nu+1}))}\le \|\varphi\|_{L^r(0,T;L^r(\cO_{\nu+1}))}^{s/(2+s)}\|\varphi\|_{L^r(0,T;W^{2+s,r}(\cO_{\nu+1}))}^{2/(2+s)}.
$$
Then, by density, taking $\varphi=(-\Delta)^{-1}(\phi_\nu\psi_k)$, where by $-\Delta$ we mean the minus Laplacian operator with 0 Dirichlet condition on $\po\cO_{\nu+1}$, we conclude that $\phi_\nu\psi_k$ strongly converges in $L^r(0,T;L^r(\cO_{\nu+1}))$, using that $(-\Delta)^{-1}$ isomorphically takes  $L^r(0,T;L^r(\cO_{\nu+1}))$ onto
$L^r(0,T; W^{2,r}\cap W_0^{1,r}(\cO_{\nu+1}))$.

 In this way, by a standard diagonal argument, we obtain a subsequence of $\psi^k$, still denoted $\psi^k$, such that $\phi_\nu\psi^k$ converges in
$C([0,T];W^{-2,r}(\R^N))\cap L^r([0,T];L^r(\R^N))$, for all $\nu\in\N$, which implies the compactness of $K_R$ in $\mathcal{X}$. 

As for the tightness of $\mu_{u^\ve}$, we have
\begin{multline*}
\mu_{u^\ve}(\mathcal{X}\setminus K_R)\le
\bbP\Big(\|u^\ve\|_{L^\infty(0,T;L_{\rho_*}^2(\R^N))}>R\Big)\\
+\sum_{\nu=1}^\infty \bbP\Big(\|\phi_\nu u^\ve\|_{C^\l([0,T];W^{-2,r}(\R^N))}>2^\nu(C_{\mathcal{O}_{\nu+1}} +1)R\Big)\\
+\sum_{\nu=1}^\infty \bbP\Big(\|\phi_\nu u^\ve\|_{L^r(0,T;W^{s,r}(\R^N))}>2^\nu(C_{\mathcal{O}_{\nu+1}} +1)R\Big)\\
\le \frac1{R^2}\bbE\sup_{[0,T]}\|u^\ve(t)\|_{L_{\rho_*}^2(\R^N)}^2\\
+\sum_{\nu=1}^\infty \frac1{2^\nu(C_{\mathcal{O}_{\nu+1}} +1)R}\bbE\left(\|\phi_\nu u^\ve\|_{C^\l([0,T];W^{-2,r}(\R^N))}+\|\phi_\nu u^\ve\|_{L^r(0,T;W^{s,r}(\R^N))}\right)\\
\le \frac{C}{R^2}+\frac2{R},
\end{multline*}
by using \eqref{eP31}, in the limit as $\mu\to0$,  \eqref{e3.10'} and \eqref{e3.100}, which implies the tightness of $\mu_{u^\ve}$.  

\end{proof}

With Proposition~\ref{P:3.2''} at hand, we apply Prokhorov's theorem to obtain a subsequence $u^n$ such that $\mu_{u^n}$ weakly converges in $\mathcal{X}$. We can then apply Skorokhod's theorem and obtain a further subsequence still denoted $u^n$, a new probability space $(\tilde \Om,\tilde\bbP)$, and a subsequence $\tilde u^n$, with $\mu_{\tilde u^n}=\mu_{u^n}$, such that $\tilde u^n:\tilde \Om\to\mathcal{X}$ converges a.s.\ to $\tilde u:\tilde \Om\to \mathcal{X}$.  

\begin{proposition}\label{P:3.2'''} There exists a probability space $(\tilde \Om,\tilde \F,\tilde \bbP)$ with a sequence of $\mathcal{X}$-valued random variables $\tilde u^n$, $n\in\N$, and
$\tilde u$ such that:
\begin{enumerate}
\item[(i)] the laws of $\tilde u^n$ and $\tilde u$ under $\tilde \bbP$ coincide with $\mu^n$ and $\mu$, respectively,
\item[(ii)] $\tilde u^n$ converges $\tilde \bbP$-almost surely to $\tilde u$ in the topology of $\mathcal{X}$.

\end{enumerate}
\end{proposition}

We then define yet another probability space $\overline{\Om}=\Om\X\tilde \Om$ with the product probability measure $\overline{\bbP}=\bbP\X\tilde{\bbP}$, the $\s$-algebra $\overline{\F}$ as the product  $\s$-algebra generated by $\tilde \F\X \F$,    and, from the Wiener process $W(t)$ in $\Om$, we define the Wiener process $\overline{W}(t)$ in $\overline{\Om}$ trivially by $\overline{W}(t)(\om,\tilde \om)=W(t)(\om)$, for $(\om,\tilde \om)\in\overline{\Om}=\Om\X\tilde\Om$; clearly, $\overline{W}$ has the same law as $W$.  Defining $\bar u^n:\overline{\Om}\to\mathcal{X}$ by $\bar u^n(\om,\tilde \om)=\tilde u^n(\tilde \om)$, we have that $\mu_{\bar u^n}=\mu_{\tilde u^n}=\mu_{u^n}$. Also, $\bar u^n$ converges a.s.\ in $\overline{\Om}$ to the random variable $\bar u:\overline{\Om}\to \mathcal{X}$ defined by $\bar u(\om,\tilde \om)=\tilde u(\tilde\om)$. We define a  filtration $\overline{\F}_t$ for $(\overline{\Om}, \overline{\F},\overline{\bbP})$ in the following way ({\em cf.} \cite{Ha}).
For each $t\in[0,T]$, the restriction map $\rho_t: C([0,T]; W_\loc^{-2,r} (\R^N)\X C([0,T];\mathfrak{U}_0)\to C([0,t]; W_\loc^{-2,r} (\R^N)\X C([0,t];\mathfrak{U}_0)$, $(v,W)\mapsto (v,W)|[0,t]$, is a continuous map. Here, $\mathfrak{U}_0$ is the Hilbert space where the cylindrical Wiener process $W(t)$ is well defined. So, we define as $\overline{\F}_t=\s(\rho_t\bar u,\rho_t)$, the $\s$-algebra of subsets of $\Om$ generated 
by the function $(\rho_t\bar u, \rho_t\overline{W}):\Om\to C([0,t];W_\loc^{-2,r}(\R^N))\X C([0,t];\mathfrak{U}_0)$, and we denote also by $\overline{\F}_t$ the corresponding augmented filtration, i.e., the smallest complete right-continuous filtration containing $\overline{\F_t}$.

\begin{definition}\label{D:3.1} We say that $\tilde u$ is a $\BAP$-entropy martingale solution of \eqref{e1.1}-\eqref{e1.2} if, for some probability space equipped with a filtration $(\bar \Om, \overline{\F}, (\overline{\F}_t), \overline{\bbP})$ and some cylindrical Wiener process $\overline{W}(t)=\sum_{i=1}^\infty \bar \b_k e_k$,
with respect to the filtration $(\overline{\F}_t)$,
with $\{e_k\}_{k\ge1}$  a complete orthonormal system in a Hilbert space $H$  and $\tilde \b_k$, $k\in\N$, independent Brownian motions in $(\bar \Om, \overline{\F}, (\overline{\F}_t), \overline{\bbP})$  
if $\bar w=\bar u-\bar J$ satisfies Definition~\ref{D:2.1} with $u$, $J$, $W$, replaced by $\bar u$, 
$\bar J$,  $\overline{W}$.
 \end{definition}

We are now going to prove that  following important fact.

\begin{theorem}\label{T:4.new}  The limit $\bar u$ with $(\bar \Om, \overline{\F}, (\overline{\F}_t), \overline{\bbP},\overline{W})$ is  a $\BAP$-entropy martingale solution in the sense of Definition~\ref{D:3.1}, that is, that it satisfies all items in Definition~\ref{D:2.1}, with the replacements referred to in Definition~\ref{D:3.1}.
\end{theorem}

 \begin{proof} Indeed, we first recall that, from Proposition~\ref{P:3.new}, the bound on 
$$
\bbE\sup\limits_{0\le t\le T}\|u^{\mu,\ve}(t)\|_{\B^2(\R^N)}^2
$$ 
given by \eqref{e3.12} holds independently of $\mu,\ve$. Therefore, in particular, a.s., $\bar u\in L^\infty([0,T]; \B^1(\R^N))$.  

Another important point is that by Proposition~\ref{P:3.1}, in the limit as $\mu\to0$, we have that a.s.\ $u^\ve\in L^2([0,T];W_{\rho_*}^{1,2}(\R^N))$ and so this also holds for $\bar u^\ve$. On the other hand, for any continuous $\g:\mathcal{X}\X\mathcal{W}\to [0,1]$, with $\mathcal{W}=C([0.1];\mathfrak{U}_0)$, and any $\phi\in C_c^\infty(\R^N)$, we have
\begin{multline*}
\bbE\g(\bar u^\ve, \overline{W})\int_0^T\int_{\R^N}\po_{x_i}\bar u^\ve \phi\,dx\,dt=
-\bbE\g(\bar u^\ve, \overline{W})\int_0^T\int_{\R^N}\bar u^\ve\po_{x_i} \phi\,dx\,dt\\
=-\bbE\g( u^\ve,  W)\int_0^T\int_{\R^N} u^\ve\po_{x_i} \phi\,dx\,dt
=\bbE\g( u^\ve,  W)\int_0^T\int_{\R^N}\po_{x_i} u^\ve \phi\,dx\,dt,
\end{multline*}
$i=1,\cdots,N$, which shows that also the derivatives $\po_{x_i}\bar u^\ve$ have the same laws as the corresponding derivatives $\po_{x_i}\bar u^\ve$. Now, for each $\ve>0$, the solution $u^\ve$ of \eqref{e4.1}-\eqref{e4.2} satisfies the following inequality, with $w^\ve=u^\ve-J^\ve$, for all $0\le \phi\in C_c^\infty((0,T)\X\R^N)$,
\begin{multline}\label{e4.entropy}
\int\limits_{(0,T)\X\R^N} \left\{\eta(w^\ve)\po_t\phi+\sum_{i=1}^N q_i(w^\ve,t,x)\po_{x_i}\phi+\sum_{i,j=1}^N r_{ij}^\ve(w^\ve,t,x)\po_{x_ix_j}^2\phi \right.\\
 \left. -\sum_{i=1}^N\left(\eta'(w^\ve)b_{i,x_i}(w^\ve,t,x)-q_{i,x_i}(w^\ve,t,x)\right)\phi+\sum_{i,j=1}^N r_{ij,x_i}^\ve(w^\ve,t,x)\po_{x_i}\phi\right\}\,dx\,dt\\
 - \int\limits_{(0,T)\X\R^N}\eta''(w^\ve)\sum_{k=1}^K\left(\sum_{i=1}^N\left(\po_{x_i}\z_{ik}^\ve(w^\ve,t,x)-\z_{ik,x_i}^\ve(w^\ve,t,x)\right)\right)^2\phi\, dx\, dt\ge 0,
 \end{multline}
where $r'(w,t,x)=\eta'(w)\abf^\ve(w+J)$, with $\abf^\ve(u)=\abf(u)+\ve I_{N\X N}$ and $I_{N\X N}$ is the $N\X N$ identity matrix, and 
$\z_{ik}^\ve(w,t,x)=\int_0^w\s_{ik}^\ve(v,t,x)\,dv$, with $\s_{ik}^\ve(v,t,x)=\bar\s_{ik}^\ve(w+J^\ve)$, and $\bar \s^\ve(u)$ is the $N\X K$ matrix such that  $\abf^\ve(u)=\bar\s^\ve(u)\bar\s^\ve(u)^\top$. Therefore, for $\bar w^\ve=\bar u^\ve-\bar J^\ve$ and 
$$
\bar J^\ve(x,t):=\sum_{k=1}^{[1/\ve]} g_k(x) \bar \b_k(t),
$$
we have,  for any continuous $\g:\mathcal{X}\X\mathcal{W}\to [0,1]$, $0\le \phi\in C_c^\infty((0,T)\X\R^N)$,
\begin{multline}\label{e4.entropy'}
\bbE\gamma(\bar u^\ve, \overline{W})\left[\int\limits_{(0,T)\X\R^N} \left\{\eta(\bar w^\ve)\po_t\phi+\sum_{i=1}^N q_i(\bar w^\ve,t,x)\po_{x_i}\phi+\sum_{i,j=1}^N r_{ij}^\ve(\bar w^\ve,t,x)\po_{x_ix_j}^2\phi \right.\right.\\
 \left. -\sum_{i=1}^N\left(\eta'(\bar w^\ve)b_{i,x_i}(\bar w^\ve,t,x)-q_{i,x_i}(\bar w^\ve,t,x)\right)\phi+\sum_{i,j=1}^N r_{ij,x_i}^\ve(\bar w^\ve,t,x)\po_{x_i}\phi\right\}\,dx\,dt\\
 -\left. \int\limits_{(0,T)\X\R^N}\eta''(\bar w^\ve)\sum_{k=1}^K\left(\sum_{i=1}^N\left(\po_{x_i}\z_{ik}^\ve(\bar w^\ve,t,x)-\z_{ik,x_i}^\ve(\bar w^\ve,t,x)\right)\right)^2\phi\, dx\, dt\right]\\
=\bbE\gamma( u^\ve, W)\left[\int\limits_{(0,T)\X\R^N} \left\{\eta( w^\ve)\po_t\phi+\sum_{i=1}^N q_i( w^\ve,t,x)\po_{x_i}\phi+\sum_{i,j=1}^N r_{ij}^\ve( w^\ve,t,x)\po_{x_ix_j}^2\phi \right.\right.\\
 \left. -\sum_{i=1}^N\left(\eta'( w^\ve)b_{i,x_i}( w^\ve,t,x)-q_{i,x_i}( w^\ve,t,x)\right)\phi+\sum_{i,j=1}^N r_{ij,x_i}^\ve( w^\ve,t,x)\po_{x_i}\phi\right\}\,dx\,dt\\
\left. - \int\limits_{(0,T)\X\R^N}\eta''( w^\ve)\sum_{k=1}^K\left(\sum_{i=1}^N\left(\po_{x_i}\z_{ik}^\ve( w^\ve,t,x)-\z_{ik,x_i}^\ve( w^\ve,t,x)\right)\right)^2\phi\, dx\, dt\right] \ge0.
 \end{multline}
Hence, 
\begin{multline}\label{e4.entropy''}
\bbE\gamma(\bar u^\ve, \overline{W})\left[\int\limits_{(0,T)\X\R^N} \left\{\eta(\bar w^\ve)\po_t\phi+\sum_{i=1}^N q_i(\bar w^\ve,t,x)\po_{x_i}\phi+\sum_{i,j=1}^N r_{ij}^\ve(\bar w^\ve,t,x)\po_{x_ix_j}^2\phi \right.\right.\\
 \left. -\sum_{i=1}^N\left(\eta'(\bar w^\ve)b_{i,x_i}(\bar w^\ve,t,x)-q_{i,x_i}(\bar w^\ve,t,x)\right)\phi+\sum_{i,j=1}^N r_{ij,x_i}^\ve(\bar w^\ve,t,x)\po_{x_i}\phi\right\}\,dx\,dt\\
\left. - \int\limits_{(0,T)\X\R^N}\eta''(\bar w^\ve)\sum_{k=1}^K\left(\sum_{i=1}^N\left(\po_{x_i}\z_{ik}^\ve(\bar w^\ve,t,x)-\z_{ik,x_i}^\ve(\bar w^\ve,t,x)\right)\right)^2\phi\, dx\, dt\right]\ge0,
 \end{multline}
for each $\g:\mathcal{X}\X\mathcal{W}\to [0,1]$,
$0\le \phi\in C_c^\infty((0,T)\X\R^N)$. Therefore, we have a.s.\ we have
\begin{multline}\label{e4.entropy'''}
\int\limits_{(0,T)\X\R^N} \left\{\eta(\bar w^\ve)\po_t\phi+\sum_{i=1}^N q_i(\bar w^\ve,t,x)\po_{x_i}\phi+\sum_{i,j=1}^N r_{ij}^\ve(\bar w^\ve,t,x)\po_{x_ix_j}^2\phi \right.\\
 \left.  \sum_{i=1}^N\left(\eta'(\bar w^\ve)b_{i,x_i}(\bar w^\ve,t,x)-q_{i,x_i}(\bar w^\ve,t,x)\right)\phi+\sum_{i,j=1}^N r_{ij,x_i}^\ve(\bar w^\ve,t,x)\po_{x_i}\phi\right\}\,dx\,dt\\
\ge  \int\limits_{(0,T)\X\R^N}\eta''(\bar w^\ve)\sum_{k=1}^K\left(\sum_{i=1}^N\left(\po_{x_i}\z_{ik}^\ve(\bar w^\ve,t,x)-\z_{ik,x_i}^\ve(\bar w^\ve,t,x)\right)\right)^2\phi\, dx\, dt.
\end{multline}
The left-hand side of \eqref{e4.entropy'''} ($L.H.S.$, for short)  is a.s.\ bounded in $L^r([0,T]\X\R^N)$, since $\bar u^\ve$ converges a.s.\ in $L_\loc^r([0,T]\X\R^N)$, with $1<r<2$. Then, for any fixed $\om\in\Om$ in a subset of total measure, $L.H.S.$ is equi-integrable and so converges weakly in $L^1((0,T)\X\R^N)$. But since $\bar u^\ve$ converges in $L_\loc^r((0,T)\X\R^N)$ to $\bar u$, we obtain that $L.H.S.$ converges as $\ve\to0$ to
\begin{multline*}
\int\limits_{(0,T)\X\R^N} \left\{\eta(\bar w)\po_t\phi+\sum_{i=1}^N q_i(\bar w,t,x)\po_{x_i}\phi+\sum_{i,j=1}^N r_{ij}(\bar w,t,x)\po_{x_ix_j}^2\phi \right.\\
 \left.  \sum_{i=1}^N\left(\eta'(\bar w)b_{i,x_i}(\bar w,t,x)-q_{i,x_i}(\bar w,t,x)\right)\phi+\sum_{i,j=1}^N r_{ij,x_i}(\bar w,t,x)\po_{x_i}\phi\right\}\,dx\,dt,
 \end{multline*}
   
As for the right-hand side of  \eqref{e4.entropy'''} ($R.H.S.$, for short), using  the fact that $\eta\in{\mathcal E}$,  we deduce that the sequence of vector functions with values in $\R^K$
$$
\left(\sum_{i=1}^N\left(\po_{x_i}\z_{ik}^\ve(\bar w^\ve,t,x)-\z_{ik,x_i}^\ve(\bar w^\ve,t,x)\right)\right)\eta''(\bar w^\ve)^{1/2}, \quad k=1,\cdots, K,
$$
is bounded in $L_\loc^2((0,T)\X\R^N)$,  and by the usual chain rule the above expression is equal to 
$$
\sum_{i=1}^N\left(\po_{x_i}\z_{ik}^{\ve,\eta''(\cdot)^{1/2}} (\bar w^\ve,t,x)-\z_{ik,x_i}^{\ve,\eta''(\cdot)^{1/2}}(\bar w^\ve,t,x)\right), \quad k=1,\cdots, K,
$$
but the latter converges in the sense of the distributions to
$$
\sum_{i=1}^N\left(\po_{x_i}\z_{ik}^{\eta''(\cdot)^{1/2}} (\bar w,t,x)-\z_{ik,x_i}^{\eta''(\cdot)^{1/2}}(\bar w,t,x)\right),
 \quad k=1,\cdots, K.
$$
Therefore, taking the $\liminf$ in \eqref{e4.entropy'''} and using the lower semicontinuity of the $L^2$-norm we arrive at 
\begin{multline}\label{e4.entropyiv}
\int\limits_{(0,T)\X\R^N} \left\{\eta(\bar w)\po_t\phi+\sum_{i=1}^N q_i(\bar w,t,x)\po_{x_i}\phi+\sum_{i,j=1}^N r_{ij}(\bar w,t,x)\po_{x_ix_j}^2\phi \right.\\
 \left.  \sum_{i=1}^N\left(\eta'(\bar w)b_{i,x_i}(\bar w,t,x)-q_{i,x_i}(\bar w,t,x)\right)\phi+\sum_{i,j=1}^N r_{ij,x_i}(\bar w,t,x)\po_{x_i}\phi\right\}\,dx\,dt\\
\ge  \int\limits_{(0,T)\X\R^N}\sum_{k=1}^K\left(\sum_{n=1}^N\left(\po_{x_i}\z_{ik}^{\eta''(\bar w)^{1/2}} (\bar w,t,x)-\z_{ik,x_i}^{\eta''(\bar w)^{1/2}}(\bar w,t,x)\right)\right)^2\phi\, dx\, dt.
\end{multline}  
When we will have proved the validity of the chain rule, we will be able to write the right-hand side as
\begin{equation}\label{e4.entropyv}
\int\limits_{(0,T)\X\R^N}\eta''(\bar w)\sum_{k=1}^K\left(\sum_{n=1}^N\left(\po_{x_i}\z_{ik} (\bar w,t,x)-\z_{ik,x_i}(\bar w,t,x)\right)\right)^2\phi\, dx\, dt.
\end{equation}
In order to prove the validity of the chain rule for $\bar u$ (i.e., $\bar w$), we observe first that by \eqref{eP31} in the limit as $\mu\to0$, applied to $u^\ve$, with $p>2$,  implies the equi-integrability in $\Om$ of 
$\sup_{0\le t\le T}\|u^\ve(t)\|_{L_{\rho_*}^2(\R^N)}$. So, we may use $\eta(w)=\frac12 w^2$ in \eqref{e4.entropy''} 
and get that the integrand inside the expectation sign of
\begin{multline*}
\bbE\gamma(\bar u^\ve, \overline{W})\left[\int\limits_{(0,T)\X\R^N} \left\{\eta(\bar w^\ve)\po_t\phi+\sum_{i=1}^N q_i(\bar w^\ve,t,x)\po_{x_i}\phi+\sum_{i,j=1}^N r_{ij}^\ve(\bar w^\ve,t,x)\po_{x_ix_j}^2\phi \right.\right.\\
 \left.\left. -\sum_{i=1}^N\left(\eta'(\bar w^\ve)b_{i,x_i}(\bar w^\ve,t,x)-q_{i,x_i}(\bar w^\ve,t,x)\right)\phi+\sum_{i,j=1}^N r_{ij,x_i}^\ve(\bar w^\ve,t,x)\po_{x_i}\phi\right\}\,dx\,dt\right]
 \end{multline*}
 is   equi-integrable in $\Om$ for all $\gamma:\mathcal{X}\X\mathcal{W}\to[0,1]$ and all $\phi\in C_c^\infty((0,T)\X\R^N)$. In particular, since the integrand is also bounded in $L^1(\Om)$, it  converges weakly in $L^1(\Om)$, for all $\gamma:\mathcal{X}\X\mathcal{W}\to[0,1]$ and all $\phi\in C_c^\infty((0,T)\X\R^N)$. As a consequence, this is also true for   the integrand inside the expectation sign of
 \begin{equation*}
\bbE\gamma(\bar u^\ve, \overline{W})\left[ \int\limits_{(0,T)\X\R^N} \sum_{k=1}^K\left(\sum_{i=1}^N\left(\po_{x_i}\z_{ik}^\ve(\bar w^\ve,t,x)-\z_{ik,x_i}^\ve(\bar w^\ve,t,x)\right)\right)^2\phi\, dx\, dt\right].
 \end{equation*}
It is also uniformly bounded in $L^1(\Om)$ and equi-integrable in $\Om$, for all $\gamma:\mathcal{X}\X\mathcal{W}\to[0,1]$ and all $\phi\in C_c^\infty((0,T)\X\R^N)$. On the other hand, since
$\z_{ik,x_i}^\ve(w,t,x)$ is Lipschitz with respect to $w$,   $\z_{ik,x_i}^\ve(\bar w^\ve,t,x)$ strongly converges to  $\z_{ik,x_i}(\bar w,t,x)$, a.s. Therefore, $\sum_{i=1}^N\po_{x_i}\z_{ik,x_i}^\ve(\bar w^\ve,t,x)$ is a.s.\ uniformly bounded in $L^\infty((0,T); L_{\rho_*}^2(\R^N))$, $k=1,\cdots,K$. 
In particular, $\sum_{i=1}^N\po_{x_i}\z_{ik,x_i}^\ve(\bar w^\ve,t,x)$ converges in the weak*-weak topology of $L^\infty((0,T); L_{\rho_*}^2(\R^N))$. Now, since it converges to 
$\sum_{i=1}^N\po_{x_i}\z_{ik,x_i}(\bar w,t,x)$ in the sense of the distributions, we have that 
$$
\sum_{i=1}^N\po_{x_i}\z_{ik,x_i}^\ve(\bar w^\ve,t,x)\wto \sum_{i=1}^N\po_{x_i}\z_{ik,x_i}(\bar w,t,x),
$$ 
in $L^\infty([0,T]; L_{\rho_*}^2(\R^N))$, $k=1,\cdots,K$. Now, concerning the chain rule, we begin 
by observing that, by the usual chain rule, we have,
 for $k=1,\cdots,K$, for all $\psi\in C_b(\R)$, 
 \begin{multline*}
 0=\bbE\g(u^\ve,W)\left[\sum_{i=1}^N\left(\po_{x_i}\z_{ik}^\ve\psi(w^\ve,t,x)-\z_{ik,x_i}^{\ve,\psi}(w^\ve,t,x)\right)\right.\\
 \left.-\psi(w^\ve)\sum_{i=1}^N\left(\po_{x_i}\z_{ik}^\ve (w^\ve,t,x)-\z_{ik,x_i}^\ve(w^\ve,t,x)\right)\right]\\
 \bbE\g(\bar u^\ve,\overline{W})\left[\sum_{i=1}^N\left(\po_{x_i}\z_{ik}^{\ve,\psi}(\bar w^\ve,t,x)-\z_{ik,x_i}^{\ve,\psi}(\bar w^\ve,t,x)\right)\right.\\
 \left.-\psi(\bar w^\ve)\sum_{i=1}^N\left(\po_{x_i}\z_{ik}^\ve (\bar w^\ve,t,x)-\z_{ik,x_i}^\ve(\bar w^\ve,t,x)\right)\right]\\
\overset{\ve\to0}{\longrightarrow} \bbE\g(\bar u,\overline{W})\left[\sum_{i=1}^N\left(\po_{x_i}\z_{ik}^{\psi}(\bar w,t,x)-\z_{ik,x_i}^{\psi}(\bar w,t,x)\right)\right.\\
 \left.-\psi(\bar w)\sum_{i=1}^N\left(\po_{x_i}\z_{ik} (\bar w,t,x)-\z_{ik,x_i}(\bar w,t,x)\right)\right],
 \end{multline*}
 by what has been said above. Therefore, a.s.,  for $k=1,\cdots,K$, for all $\psi\in C_b(\R)$, we have
 \begin{equation*}
 \sum_{i=1}^N\left(\po_{x_i}\z_{ik}^{\psi}(\bar w,t,x)-\z_{ik,x_i}^{\psi}(\bar w,t,x)\right)
 =\psi(\bar w)\sum_{i=1}^N\left(\po_{x_i}\z_{ik} (\bar w,t,x)-\z_{ik,x_i}(\bar w,t,x)\right),
 \end{equation*}
 that is, the chain rule holds.  In particular, we may write the right-hand side of \eqref{e4.entropyiv}
 as \eqref{e4.entropyv}. Thus far we have proved the validity of items (iv) and (v) of Definition~\ref{D:2.1}.  Item (iii) also follows from what has been said above. As for item (ii), from Proposition~\ref{P:3.1}, in the limit as $\mu\to0$, we deduce that $\bar u^\ve$ is a.s.\ uniformly bounded in $L^\infty((0,T); L_{\rho_*}^2(\R^N))$, which implies that $\bar u\in L^\infty((0,T); L_{\rho_*}^2(\R^N))$. Since we have already proved that a.s.\ $\bar u\in L^\infty([0,T]; \B^1(\R^N))$,
 item (ii) follows.  As for (i) in Definition~\ref{D:2.1}, we recall that from   \eqref{e4.entropyiv}, with
 $\eta(w)=\pm w$ we can obtain the integral form of of the weak notion of solution of \eqref{e1.1}. 
 Also,    in   \eqref{e4.entropyiv} we can consider $\phi\in C_c^\infty((-\infty,T)\X\R^N)$, as long as we add the term
 $$
 \int_{\R^N} \eta(u_0(x))\phi(0,x)\,dx,
 $$
 to the  left-hand side of  \eqref{e4.entropyiv}. In particular, when $\eta(w)=\pm w$, we obtain from
 \eqref{e4.entropyiv} the weak formulation including the initial condition
 \begin{multline}\label{e4.1new2}
\int_{(0,T)\X\R^N}\left\{ \bar w\phi_t+ \fbf (\bar w+\bar J)\nabla \phi+ \Abf(\bar w+\bar J):D^2\phi\right\}\,dx\,dt\\
+\int_{\R^N} u_0(x)\phi(0,x)\,dx= 0.
\end{multline} 
{}From \eqref{e4.1new2}, taking $\phi(t,x)=\frac{1}{R^N}g(\frac{x}{R})\varphi(t,x)$, with $g$ as in the proof of Proposition~\ref{P:2.1}, and $\varphi\in C^\infty([0,T]; \AP^2(\R^N))$, and $\AP^2(\R^N)$ denoting the almost periodic function in $\AP(\R^N)$ whose derivatives up to order 2 also belong to $\AP(\R^N)$, we obtain 
 \begin{multline}\label{e4.1new3}
\int_{(0,T)\X\bbG_N}\left\{ \bar w\varphi_t+ \fbf (\bar w+\bar J)\nabla \varphi+ \Abf(\bar w+\bar J):D^2\varphi\right\}\,d\mm(x)\,dt\\
+\int_{\bbG_N} u_0(x)\varphi(0,x)\,d\mm(x)= 0.
\end{multline} 
{}From  \eqref{e4.1new2}, taking $\phi(t,x)=\z^h(t)\psi(x)$  and from \eqref{e4.1new3}, taking $\varphi(t,x)=
\z^h(t)\tilde \psi(x)$,  for $\psi\in C_c^\infty(\R^N)$, $\tilde \psi\in \AP^2(\R^N)$, and a  suitable 
sequence $\z^h\in C_c^\infty([0,T))$,  approaching the indicator function of the interval 
$[t_0,t_0+\d)$, $\d>0$, $t_0\in[0,T)$,  we deduce the weak continuity required in (i) of Definition~\ref{D:2.1}. 

Finally, concerning (vi) of Definition~\ref{D:2.1}, it also follows in a standard way, using an idea in \cite{Ot} (see also \cite{MNRRO}), departing from  
\begin{multline}\label{e4.entropyiv'}
\int\limits_{(0,T)\X\R^N} \left\{\eta(\bar w)\po_t\phi+\sum_{i=1}^N q_i(\bar w,t,x)\po_{x_i}\phi+\sum_{i,j=1}^N r_{ij}(\bar w,t,x)\po_{x_ix_j}^2\phi \right.\\
 \left.  -\sum_{i=1}^N\left(\eta'(\bar w)b_{i,x_i}(\bar w,t,x)-q_{i,x_i}(\bar w,t,x)\right)\phi+\sum_{i,j=1}^N r_{ij,x_i}(\bar w,t,x)\po_{x_i}\phi\right\}\,dx\,dt\\
 +\int_{\R^N}\eta(u_0(x))\phi(0,x)\,dx\ge 0,
\end{multline}    
with $\eta(w)=|w-c|$, $c\in\R$, $\phi\in C_c^\infty((-\infty,T)\X\R^N)$, choosing $\phi(t,x)=\z^h(t)\psi(x)$, for $\psi\in C_c^\infty(\R^N)$ and $\z^h(t)$ a suitable sequence approaching the indicator function of $[0,\d)$, $\d>0$, etc. 
\end{proof}

Having verified that $\bar u$ satisfies all the conditions of Definition~\ref{D:2.1}, we then have that 
$\bar u$ enjoys the $L_{\rho_*}^1(\R^N)$-stability given by Theorem~\ref{T:2.1}. Therefore, we can apply  the Gyongy-Krylov criterion for convergence in probability, introduced in \cite{GH}, to conclude that the whole sequence $u^\ve$ converges strongly in $L_\loc^1((0,T)\X\R^N)$
to an entropy solution of \eqref{e1.1}-\eqref{e1.2}.

We close this section by stating the definition of $L^1(\bbG_N)$-entropy solution and the companion notion of semi-group solution. 

\begin{definition}\label{D:4.10} Let $u_0\in L^1(\bbG_N)$ and  $T>0$ be given. A $L^1(\bbG_N)$-valued stochastic process,  adapted
to $\{\F_t\}$, is said to be a $L^1(\bbG_N)$-entropy solution of \eqref{e1.1}-\eqref{e1.2} if, for almost all $\om\in\Om$, for $w=u-J$,
\begin{enumerate}
 \item[(i)] $w(t)$ is $L^1(\bbG_N)$-weakly continuous on $[0,T]$.
 \item[(ii)] $w\in L^\infty([0,T];  L^1(\bbG_N))$.  
 \item[(iii)] (Weak regularity) For $k=1,\cdots,K$,
 $$
 \sum_{i=1}^N\left(\po_{x_i}\z_{ik}(w,t,z)-\z_{ik,x_i}(w,t,z)\right)\in L^2((0,T)\X\bbG_N).
 $$
 \item[(iv)] (Chain Rule) For $k=1,\cdots,K$,
 $$
 \sum_{i=1}^N\left(\po_{x_i}\z_{ik}^\psi(w,t,z)-\z_{ik,x_i}^\psi(w,t,z)\right)=\psi(w)\sum_{i=1}^N\left(\po_{x_i}\z_{ik}(w,t,z)-\z_{ik,x_i}(w,t,z)\right)
 $$
 a.e.\ in $(0,T)\X\bbG_N$, for any $\psi\in C_b(\R)$.
 
 \item[(v)] (Entropy Inequality) For any entropy-entropy flux triple $(\eta,q,r)$, with $\eta\in\mathcal{E}$, for all $0\le \varphi\in C_c^\infty([0,T)\X\bbG_N)$,
 \begin{multline}\label{e40.entropy}
 \int_0^T\int_{\bbG_N}\left\{(\eta(w)-\eta(u_0))\varphi_t+\sum_{i=1}^N q_i(w,t,z)\po_{x_i}\varphi +\sum_{i,j=1}^N r_{ij}(w,t,z)\po_{x_ix_j}^2\varphi\right. \\
\left.-\sum_{i=1}^N\left(\eta'(w)b_{i,x_i}(w,t,z)-q_{i,x_i}(w,t,z)\right)\varphi
+\sum_{i,j=1}^Nr_{ij,x_i}(w,t,z)\po_{x_i}\varphi\right\}\,d\mm(z)\,dt\\
 \ge \int_0^T\int_{\bbG_N}\eta''(w)\sum_{k=1}^K\left(\sum_{i=1}^N\left(\po_{x_i}\z_{ik}(w,t,z)-\z_{ik,x_i}(w,t,z)\right)\right)^2\varphi\,d\mm(z)\,dt.
 \end{multline}
 \end{enumerate}
 \end{definition}

 \begin{theorem}\label{T:3new} Given, $u_0\in L^1(\bbG_N)$, there exists a $L^1(\bbG_N)$-entropy solution of  \eqref{e1.1}-\eqref{e1.2} in the sense of Definition~\ref{D:4.10} which is the  $L^1(\Om; L^\infty((0,T);L^1(\bbG_N)))$ limit of
$\BAP$-entropy solutions of \eqref{e1.1}-\eqref{e1.2} with initial data in $\AP(\R^N)$. 
Moreover, given two $L^1(\bbG_N)$-entropy solutions $u(t,x), v(t,x)$, both of which are obtained as limits in $L^1(\Om; L^\infty((0,T);L^1(\bbG_N)))$ of $\BAP$-entropy solutions of
\eqref{e1.1}-\eqref{e1.2}, with initial functions converging to $u_0, v_0$ in $L^1(\bbG_N)$, 
 then a.s.\  it holds
\begin{equation}\label{e4.semicont} 
\int_{\bbG_N}|u(t,z)-v(t,z)|\,d\mm(z)\le C(T) \int_{\bbG_N}|u_0(z)-v_0(z)|\, d\mm(z),
\end{equation}
where $C(T)$ is as in \eqref{e2.4AP}. 
\end{theorem}

\begin{proof} Indeed, existence of a $L^1(\bbG_N)$-entropy solution of \eqref{e1.1}-\eqref{e1.2}  in the sense of Definition~\ref{D:4.10},
when $u_0\in\AP(\R^N)$,  follows from the existence of $\BAP$-entropy solution of \eqref{e1.1}-\eqref{e1.2} in the sense of Definition~\ref{D:2.1}, proved above,  by applying  \eqref{e2.entropy} to a test function of the form
$$
\varphi(t,x)= R^{-N}g(x/R) \tilde\varphi(t,x),
$$ 
where $g$ is as in the proof of Proposition~\ref{P:2.1} and $0\le \varphi\in C_c^\infty([0,T); \AP^\infty(\R^N))$. In this way we obtain that the $\BAP$-entropy solution satisfies \eqref{e4.entropy} of Definition~\ref{D:4.10} for all nonnegative $\tilde\varphi\in C_0^\infty ((0,T)\X\bbG_N)$ and $\eta\in\mathcal{E}$. We can then extend this to test functions $0\le \varphi\in C_0^\infty ([0,T)\X\bbG_N)$ by taking in the inequality already obtained  for all nonnegative $\tilde\varphi\in C_0^\infty ((0,T)\X\bbG_N)$ a test function of the form $\tilde\varphi(t,x)=\d_h(t)\varphi(t,x)$, with 
 $0\le \varphi\in C_0^\infty ([0,T)\X\bbG_N)$  and $\d_h$ a suitable sequence in $C_c^\infty((0,T))$ converging  everywhere in $(0,T)$ to the indicator function of the interval $(0,T)$, which proves (v) in Definition~\ref{D:4.10}. 
 
 We next observe that \eqref{e40.entropy} implies
  \begin{multline}\label{e40.entropy'}
 \int_0^T\int_{\bbG_N}\left\{(\eta(w)-\eta(u_0))\varphi_t+\sum_{i=1}^N q_i(w,t,z)\po_{x_i}\varphi +\sum_{i,j=1}^N r_{ij}(w,t,z)\po_{x_ix_j}^2\varphi\right. \\
\left.-\sum_{i=1}^N\left(\eta'(w)b_{i,x_i}(w,t,z)-q_{i,x_i}(w,t,z)\right)\varphi
+\sum_{i,j=1}^Nr_{ij,x_i}(w,t,z)\po_{x_i}\varphi\right\}\,d\mm(z)\,dt\ge 0,
\end{multline}
which, by approximation can be extended for $\eta$ of the form $\eta(w)=|w-\a|$, $\a\in\R$, with
$q(w,t,x)=\sgn(w-\a)(b(w,t,x)-b(\a,t,x))$, $r(w,t,x)=\sgn(w-\a)(a(w,t,x)-a(\a,t,x))$ and $a(w,t,x),\, b(w,t,x)$ given by \eqref{e2.1new}. 
We then  observe that making $\a\to+\infty$ in \eqref{e40.entropy'}, with $\eta(w)=|w-\a|$,  splitting the integral over $\bbG_N$ into two parts, one over $\{w\le \a\}$ and other over 
$\{w>\a\}$, and, after, also making $\a\to-\infty$ and proceeding similarly, we obtain the following integral equation
\begin{multline}\label{e40.3eq}
 \int_0^T\int_{\bbG_N} \{w\varphi_t+\bff(w+J)\cdot\nabla\varphi +\Abf(w+J):D^2\varphi\}\,d\mm(z)\,dt
 \\
 =\int_{\bbG_N} u_0(z) \varphi(0,z )\,d\mm(z),
 \end{multline}
for all $\varphi\in C_0^\infty ([0,T)\X\bbG_N)$. The fact that \eqref{e40.3eq} holds for all $\varphi\in C_0^\infty ([0,T)\X\bbG_N)$ implies, in turn, in a standard way, the item (i) of Definition~\ref{D:4.10}. 

We may verify the item~(ii) of Definition~\ref{D:4.10} also as a consequence of item (v) of Definition~\ref{D:4.10}. Indeed, from \eqref{e40.entropy'} with $\eta(w)=|w|$ and $\varphi(t,z)=\d_h(t)$, where the latter is a sequence in $C_c^\infty((0,T))$ converging everywhere to $(0,t_0)$, and $t_0$ is any Lebesgue point of $\int_{\bbG_N}|w(t)|\,d\mm(z)$, we obtain
\begin{equation}\label{e3.new}
\int_{\bbG_N}|w(t_0,z)|\,d\mm(z)\le \int_{\bbG_N}|u_0|\,d\mm(z) + \int_0^T\int_{\bbG_N} C(J, \nabla J, \nabla^2J)\,d\mm(z)\,dt,
\end{equation}
where $C(J,\nabla J, \nabla^2 J)$ is a continuous function of $T$,  $J$ and its derivatives up to the second order. In particular, a.s.,  $w\in L^\infty([0,T];  L^1(\bbG_N))$, which  proves (ii). 

Item (iii) also follows from \eqref{e40.entropy} by taking $\varphi$, by approximation, in the form $\varphi(t,z)={\bf 1}_{[0,T)}$ for  $\eta$ in a sequence in ${\mathcal E}$ converging everywhere in $\R$ to $\frac12
|w|^2$ such that $\eta''(w)\to 1$, everywhere in $\R$,  which gives, since we are assuming $u_0\in\AP(\R^N)$, 
\begin{multline}\label{e4.100}
 \int_0^T\int_{\bbG_N}\sum_{k=1}^K\left(\sum_{i=1}^N\left(\po_{x_i}\z_{ik}(w,t,z)-\z_{ik,x_i}(w,t,z)\right)\right)^2\,d\mm(z)\,dt \\
 \le \int_{\bbG_N}|u_0(z)|^2d\mm(z) + C\int_0^T\int_{\bbG_N}|u(t,z)|^2d\mm(z)dt 
\end{multline}
where $C>0$ depends on $T$, $J$, $\nabla J$ and $D^2J$. Now, since the bound from Proposition~\ref{P:3.new} is uniform in $\mu$ and $\ve$, then \eqref{e3.12} holds for $u$ and the right hand side of \eqref{e4.100} is finite a.e., which implies the validity of item (iii). Finally, item (iv) follows directly from (iv) of Definition~\ref{D:2.1}, by multiplying that equality by a test function of the form 
$\tilde\varphi(t,x)=1/R^Ng(x/R) \varphi(t,x)$, where $g$ is as in the proof of Proposition~\ref{P:2.1}, $\varphi\in C_c((0,T)\X\bbG_N)$, making $R\to\infty$,  
which gives the equality in the form of an integral equation  valid for all $\varphi\in C_c((0,T)\X\bbG_N)$, which in turn proves (iv) of Definition~\ref{D:4.10}, and co  
  concludes the proof that $\BAP$-entropy solutions are also (or extend to) $L^1(\bbG_N)$-entropy solutions. 

{}Also, given two  $\BAP$-entropy
solutions $u(t,x), v(t,x)$, with initial data $u_0, v_0\in \AP(\R^N)$, we get from  \eqref{e2.4AP}
\begin{equation}\label{e4.10} 
\int_{\bbG_N} |u(t,z)-v(t,z)|\,d\mm(z)\le C(T)\int_{\bbG_N}|u_0(z)-v_0(z)|\,d\mm(z).
\end{equation}
Now, from \eqref{e4.10} we can extend the existence of $L^1(\bbG_N)$-entropy solutions for initial data $u_0\in L^1(\bbG_N)$. Indeed,  if we approximate the initial data $u_0\in L^1(\bbG_N)\sim \B^1(\R^N)$ in $\B^1(\R^N)$ by a sequence $u_{0n}\in \AP(\R^N)$, from \eqref{e4.10} we deduce that the corresponding $\BAP$-entropy solutions $u_n$, in the sense of Definition~\ref{D:2.1}, form a Cauchy sequence in $L^\infty((0,T); L^1(\bbG_N))$, and so, there is $u\in L^\infty((0,T);L^1(\bbG_N))$ such that $u_n\to u$ in $L^\infty((0,T);L^1(\bbG_N))$. This is true a.s.\ in $\Om$ and since by \eqref{e3.new} the norm of the $u_n$'s in $L^\infty((0,T);L^1(\bbG_N))$  is
 bounded by a function in $L^1(\Om)$, we conclude by dominated convergence that 
$$
u_n\to u \quad \text {in $L^1(\Om;L^\infty((0,T); L^1(\bbG_N)))$}.
$$ 
It is then easy to check that the limit $u$ is indeed a $L^1(\bbG_N)$-entropy solution in the sense of Definition~\ref{D:4.1}. Moreover, the contraction property \eqref{e4.10}  extends to any pair of such  $L^1(\bbG_N)$-entropy solutions with initial data in $L^1(\bbG_N)$, obtained as limit in $L^1(\Om; L^\infty([0,T]; L^1(\bbG_N)))$ of  $\BAP$-entropy solutions. In this way we have proved the existence of a $L^1(\bbG_N)$-entropy solution to \eqref{e1.1}-\eqref{e1.2} for any initial data in $L^1(\bbG_N)$. The proof of \eqref{e4.semicont} follows from what has already been seen. 
\end{proof}

 \begin{definition}\label{D:4.2} Let $T>0$ be given. A $L^1(\bbG_N)$-entropy solution of \eqref{e1.1}-\eqref{e1.2} is said to be a $L^1(\bbG_N)$-semigroup solution if it is the limit in $L^1(\Om;L^\infty((0,T);L^1(\bbG_N)))$ of a sequence of $\BAP$-entropy solutions of problems like \eqref{e1.1}-\eqref{e1.2}, with initial functions converging to $u_0$ in $L^1(\bbG_N)$. 
 \end{definition}
 
 The following proposition justifies the introduction of the notion of $L^1(\bbG_N)$-semigroup solution of \eqref{e1.1}-\eqref{e1.2}. The proof is straightforward. 
 
 \begin{proposition}\label{P:4.1} Let $u$ and $v$ be two $L^1(\bbG_N)$-semigroup solutions of \eqref{e1.1}-\eqref{e1.2} with initial functions $u_0,v_0\in L^1(\bbG_N)$. Then \eqref{e4.semicont} holds.
 \end{proposition}

\section{Reduction to the periodic case}\label{S:5}

In this section and the next one we consider solutions of \eqref{e1.1}-\eqref{e1.2} taking values in a separable subspace of $L^1(\bbG_N)$. More specifically, let  $\Lambda=\{\l_j\,:\, j=1,\cdots,P\}$  a finite set of vectors in $\R^N$ linearly independent over $\Z$. We consider the  closed real algebra generated by 1 and the complex trigonometrical functions $e^{\pm i 2\pi\l_j\cdot x}$, $j=1,\cdots, P$. Any function $g$  belonging  to this algebra has spectrum $\Sp(g)$ contained in the smallest additive group generated by $\Lambda$, which we denote $\G_\Lambda$. It is not difficult to see that this closed algebra is the closed subspace of $\AP(\R^N)$ formed by the functions of the form $g(y(x))$, with $g\in C(\bbT^P)$, where $\bbT^P$ is the $P$-dimensional torus and $y(x):=(\l_1\cdot x, \cdots, \l_P\cdot x)$. Indeed, it is the completion in the $\sup$-norm of the real trigonometric polynomials of the form 
$$
s(y(x))=\sum_{\bar k\in {\mathfrak F}}a_{\bar k}e^{i2\pi \bar k\cdot y(x)},
$$
where ${\mathfrak F}\subset\Z^P$ is a finite set. Since we are considering real trigonometric polynomials, this means that ${\mathfrak F}$ should be symmetric, that is ${\mathfrak F}=-{\mathfrak F}$, and $a_{-\bar k}=\bar a_{\bar k}$, where, as usual, $\bar a_{\bar k}$ denotes the complex conjugate of $a_{\bar k}$.  Since the completion in the $\sup$-norm of the trigonometric polynomials 
$$
s(y)=\sum_{\bar k\in \mathfrak F}a_{\bar k}e^{i2\pi \bar k\cdot y},
$$
is exactly $C(\bbT^P)$, the assertion follows.

We henceforth denote by $\AP_*(\R^N)$ this subspace of $\AP(\R^N)$ and we will assume that  the noise functions $g_k$, $k\in\N$, belong to $\AP_*(\R^N)$.   

By a well known extension of the Stone-Weierstrass theorem (see, e.g,  \cite{DS}, p.274--276, Theorem~18 and Corollary~19) we have that $\AP_*(\R^N)$ is isometrically isomorphic with $C(\bbG_{*N})$, where $\bbG_{*N}$ is the topological subgroup of the Bohr compact $\bbG_N$, whose topology is generated by $\{e^{\pm 2\pi i\l_j\cdot x}\,:\, j=0,1,2,\cdots,P\}$. 
We denote by $\B_*^1(\R^N)$ the completion of $\AP_*(\R^N)$ with respect to the semi-norm $N_1$ defined in beginning of Section~\ref{S:2}. Therefore, we have that $\B_*^1(\R^N)$ is isometrically isomorphic with 
$L^1(\bbG_{*N})$.

 For simplicity, let us first consider the situation where we have as the initial data $u_0$ in \eqref{e1.1}-\eqref{e1.2} a trigonometrical polynomial. So, for some finite symmetric set ${\mathfrak F}\subset\Z^P$ as above, with $a_{-\bar k}=\bar a_{\bar k}$,
   $u_0$ can be written as
\begin{equation}\label{e4.0}
u_0(x)=\sum_{\bar k\in {\mathfrak F}} a_{\bar k}e^{2\pi i\bar k\cdot y(x)}.
\end{equation}
Therefore, $u_0(x)=v_0(y(x))$ where
\begin{equation}\label{e4.0'}
v_0(y)=\sum_{\bar k\in {\mathfrak F}} a_{\bar k}e^{2\pi i \bar k\cdot y}
\end{equation}
also,  $g_k(x)=h_k(y(x))$, with $h_k\in C(\bbT^P)$, 
and, as defined above, $y(x)=(y_1(x),\cdots,y_P(x))$, with
\begin{equation}\label{e4.0''}
y_j(x)=\l_j\cdot x=\sum_{l=1}^n\l_{jl}x_l,\ \l_j=(\l_{j1},\cdots,\l_{jN}).
\end{equation}

Consider the equation
\begin{equation}\label{e5.1}
v_t+\div_y \tilde\bff(v)=\nabla_y\cdot (\tilde \abf(v)\nabla_y v)+ \tilde \Phi\,d\tilde W,\quad \ \tilde \Phi\,d\tilde W=\sum_{k=1}^\infty h_k(y)\,d\b_l,
\end{equation}
with $\tilde \bff(v)=(\tilde f_1(v),\cdots,\tilde f_P(v))$,
$$
\tilde f_j(v)=\l_j\cdot \bff(v)=\sum_{\ell=1}^N\l_{j\ell} f_\ell(v),\quad j=1,\cdots,P,
$$
and $\tilde \abf(v)=\tilde \s(v)\tilde \s(v)^\top$,
 where $\tilde \s(v)$ is the $N\X P$ matrix defined in terms of its columns by 
 $\tilde \s(v)=[\tilde \s_1(v), \cdots, \tilde\s_P(v)]$  
 where
 $$
\tilde \s_j(v):=\l_j^\top \s(v),\quad j=1,\cdots,P.
$$
The Cauchy problem in $\bbT^P$ for \eqref{e5.1} is formed by prescribing the initial datum
\begin{equation}\label{e5.1'}
v(0,y)=v_0(y), \quad y\in\bbT^P.
\end{equation}
Setting $\tilde w=v-\tilde J$, with
$$
\tilde J(t,y):=\sum_{l=1}^\infty h_l(y) \b_k(t),
$$
we can write \eqref{e5.1} as
\begin{equation}\label{e4.1''}
\tilde w_t+\div_y (\tilde\bff(\tilde w+\tilde J)+\tilde\abf(\tilde w+\tilde J)\nabla\tilde J)-
\nabla_y\cdot (\tilde {\abf}(\tilde w+\tilde J)\nabla_y \tilde w)=0,
\end{equation}
and \eqref{e5.1'} as
\begin{equation}\label{e4.1'''}
\tilde w(0,y)=v_0(y),\quad y\in\bbT^P.
\end{equation}
We can then define entropy solution for the periodic problem \eqref{e5.1}-\eqref{e5.1'} in a way entirely analogous to Definition~\ref{D:2.1}.
Let us then denote 
\begin{align*}
& \tilde\z_{ik}(\tilde w,t,y)=\int_0^{\tilde w}\tilde\s_{ik}(v,t,y)\,dv,\\
&\tilde \z_{ik}^\psi(\tilde w,t,y)=\int_0^{\tilde w}\psi(v)\tilde\s_{ik}(v,t,y)\, dv,\quad \text{for $\psi\in C(\R)$},\\
&\qquad\qquad i=1,\cdots,P,\quad k=1,\cdots, N.
\end{align*}

\begin{definition}\label{D:4.0} Let $v_0\in L^1(\bbT^P)$ and  $T>0$ be given. A $L^1(\bbT^P)$-valued stochastic process,  adapted
to $\{\F_t\}$, is said to be an entropy solution of \eqref{e5.1}-\eqref{e5.1'} if, for almost all $\om\in\Om$, for $\tilde w=v-\tilde J$,
\begin{enumerate}
 \item $\tilde w\in L^1(\bbT^P)$-weakly continuous on $[0,T]$,
 \item $\tilde w\in L^\infty([0,T]; L^1(\bbT^P))$,  
 \item (Weak regularity)
 $$
 \sum_{i=1}^P\left(\po_{y_i}\tilde \z_{ik}(\tilde w,t,y)-\tilde \z_{ik,y_i}(\tilde w,t,y)\right)\in L^2((0,T)\X\bbT^P), \quad  k=1,\cdots, N.
 $$
 \item (Chain Rule) For $k=1,\cdots,N$,
 $$
 \sum_{i=1}^P\left(\po_{y_i}\tilde \z_{ik}^\psi(\tilde w,t,y)-\tilde\z_{ik,y_i}^\psi(\tilde w,t,y)\right)=\psi(\tilde w)\sum_{i=1}^P\left(\po_{y_i}\tilde\z_{ik}(\tilde w,t,y)-\tilde\z_{ik,y_i}(\tilde w,t,y)\right)
 $$
 a.e.\ in $(0,T)\X\bbT^P$, for any $\psi\in C(\R)$.
 
 \item (Entropy Inequality) For any entropy-entropy flux triple $(\eta,q,r)$, 
 \begin{multline}\label{e4.entropy-p}
 \po_t\eta(\tilde w)+\sum_{i=1}^P\po_{y_i}q_i(\tilde w,t,y)-\sum_{i,j=1}^P\po_{y_iy_j}^2r_{ij}(\tilde w,t,y)\\
 +\sum_{i=1}^P\left(\eta'(\tilde w)b_{i,y_i}(\tilde w,t,x)-q_{i,y_i}(\tilde w,t,y)\right)+\sum_{i,j=1}^P\po_{y_i}r_{ij,y_i}(\tilde w,t,x)\\
 \le -\eta''(\tilde w)\sum_{k=1}^N\left(\sum_{i=1}^P\left(\po_{y_i}\tilde\z_{ik}(\tilde w,t,y)-\tilde\z_{ik,y_i}(\tilde w,t,y)\right)\right)^2\quad \text{in $\DD'((0,T)\X\bbT^P)$}.
 \end{multline}
 
 \item (Initial Condition) 
\begin{equation}\label{e2.D21''}
\lim_{t\to0+}\intl_{\bbT^P}|v(t,y)-v_0(y)|\,dy=0,
\end{equation}
 \end{enumerate}
 \end{definition}

Existence and uniqueness of a periodic entropy solution of \eqref{e5.1}-\eqref{e5.1'} can be proved in a way similar to what was done for the almost periodic case, but we need to impose a further non-degeneracy condition, similar to \eqref{e3.nondeg1},  with the additive group $\Z^N$ replaced by $\G_\Lambda$, so wherever
we have $n\in\Z^N$ in \eqref{e3.nondeg1} we replace it by $\b\in\G_\Lambda$.
So, the $\G_\Lambda$-symbol  is defined by
$$
\LL(i\tau,i \b, \xi):=i(\tau+b(\xi)\cdot \b) + {\b}^\top \bfa(\xi)\b,
$$  
where $b(\xi)=\bff'(\xi)$ and  $\b\in\G_\Lambda$.  
For $J,\d>0$ and $\eta\in C_b^\infty(\R)$ nonnegative, let 
\begin{equation*}
\begin{aligned}
\Om_{\LL}^\eta(\tau,\eta;\d)&:=\{\xi\in \supp\eta\,:\, |\LL(i\tau, i \b, \xi)|\le \d\},\\
\om_{\LL}^\eta(J;\d) &:= \sup_{\tiny\begin{matrix} \tau\in\R, \b\in \G_\Lambda\\ |\b|\sim J\end{matrix}}|\Om_{\LL}^\eta(\tau,i \b;\d)|,
\end{aligned}
\end{equation*}
where, for $\b=m_1\l_1+\cdots+m_P\l_P\in\G_\Lambda$, $m_j\in\Z$, $j=1,\cdots,P$,  we define $|\b|:=
\sqrt{m_1^2+\cdots+m_P^2}$.
Let $\LL_{\xi}:=\po_\xi\LL$. We suppose that there exist $\a\in (0,1)$, $\g>0$ and a measurable function  $\vartheta\in L_\loc^\infty(\R;[1,\infty))$ such that
\begin{equation}\label{e3.ndeg-new}
\begin{aligned}
\om_{\cL}^\eta(J;\d) &\lesssim_\eta \left(\frac{\d}{J^\g}\right)^\a,\\
\sup_{\tiny{\begin{matrix}\tau\in\R,\b\in\G_\Lambda \\ |\b|\sim J \end{matrix}}}\sup_{\xi\in\supp \eta}\frac{|\cL_{\xi}(i\tau,i\b;\xi)|}{\vartheta(\xi)}&\lesssim_\eta J^\g,\qquad \forall \d>0,\, J\gtrsim 1.
\end{aligned}
\end{equation}
This condition guarantees that \eqref{e5.1} enjoys a non-degeneracy condition in $\R^P$ similar to \eqref{e3.nondeg1}.

We recall the concepts of kinetic measure and of kinetic solution \eqref{e5.1}-\eqref{e5.1'} in the $L^1$ periodic setting from \cite{GH}.

\begin{definition}[Periodic kinetic measure ({\em cf.} \cite{GH})] \label{D:4.1} A map $m$ from $\Om$ to the set of non-negative Radon measures over $[0,T]\X\bbT^P\X\R$ is a kinetic measure if
\begin{enumerate}
\item $m$ is measurable, that is, for all $\phi\in C_c([0,T]\X\bbT^P\X\R)$, $\la m,\phi\ra:\Om\to \R$
is a measurable function;

\item $m$ vanishes for large $\xi$ in the sense that 
\begin{equation}\label{e5.2} 
\lim_{\ell\to\infty} \frac1{2^\ell} \bbE m(A_{2^\ell})=0,
\end{equation}
where $A_{2^\ell}=[0,T]\X\bbT^P \X\{\xi\in\R\,:\,\ 2^\ell\le |\xi|\le 2^{\ell+1}\}$,
\item for all $\phi\in C_c(\bbT^P\X\R)$, the process
$$
t\mapsto \int_{[0,t]\X\bbT^P\X\R}\phi(y,\xi)\, dm(s,y,\xi)
$$
is predictable.

\end{enumerate}
\end{definition}

\begin{definition}[Periodic kinetic solution ({\em cf.} \cite{GH}) ]\label{D:4.2P} 
Let $v_0\in L^1(\bbT^P)$. A function $v\in L^1(\Om\X[0,T],\cP, d\bbP\otimes dt; L^1(\bbT^P))$
is called a kinetic solution of  \eqref{e5.1}-\eqref{e5.1'}, where $\cP$ denotes the predictable $\s$-algebra, if
\begin{enumerate}
\item[(i)] For all $\phi\in C_c^\infty(\R)$, $\phi\ge0$,
$$
\div \int_0^v\phi(\z)\tilde \s(\z)\,d\z\in L^2(\Om\X[0,T]\X\bbT^P),
$$
where the divergence in $\R^P$ of a $N\X P$ matrix  means the $N$-vector resulting of the divergence of each of its $N$ lines.
\item[(ii)]({\em cf.} \cite{CP}) For all $\phi_1,\phi_2\in C_c^\infty(\R)$, $\phi_1,\phi_2\ge0$, the following chain
rule formula holds true in $L^2(\Om\X[0,T]\X\bbT^P)$,
\begin{equation}\label{e4.3}
\div \int_0^v\phi_1(\z)\phi_2(\z)\tilde \s(\z)\,d\z=\phi_1(v)\div\int_0^v\phi_2(\z)\tilde\s(\z)\,d\z.
\end{equation}
\item[(iii)]  Let $\phi\in C_c^\infty(\R)$, $\phi\ge0$, and let $n^\phi:\Om\to\M^+([0,T]\X\bbT^P)$ be defined as
follows: for all $\varphi\in C_c^\infty([0,T]\X\bbT^P)$,
\begin{equation}\label{e4.4}
n^\phi(\varphi)=\int_0^T\int_{\bbT^P}\varphi(t,x)\left|\div\int_0^v\sqrt{\phi(\z)}\,\tilde \s(\z)\,d\z\right|^2\,dx\,dt.
\end{equation}
There exists a kinetic measure $m$ such that, for all $\varphi\in C_c^\infty([0,T]\X\bbT^P)$, $\varphi\ge0$
and $\phi\in C_c^\infty(\R)$, $\phi\ge0$, it holds $m(\varphi \phi)\ge n^\phi(\varphi)$, $\bbP$-a.s., and, in addition, if $f(\om,t,x,\xi)={\bf 1}_{v(\om,t,x)>\xi}$,  the pair $(f, m)$ satisfies, for all $\varphi\in C_c^\infty([0,T]\X\bbT^P\X\R)$, $\bbP$-a.s., 
\begin{multline}\label{e4.5} 
\int_0^T\la f(t),\po_t\varphi (t)\ra\,dt +\la f_0,\varphi(0)\ra +\int_0^T\la f(t),\tilde b\cdot \nabla\varphi(t)\ra\,dt\\
+\int_0^T\la f(t), \tilde \Abf: D^2\varphi(t)\ra\,dt\\
=-\sum_{k\ge1}\int_0^T\int_{\bbT^P}h_k(y)\varphi(t,y,v(t,y))\,dy\,d\b_k(t)+\\
-\frac{1}{2}\int_0^T\int_{\mathbb{T}^P}\sum_{k\ge 1}|h_k(x)|^2\partial_\xi\varphi(t,y,v(t,y))\, dy\, dt +  m(\po_\xi\varphi),
\end{multline}
where $\tilde b(\xi)=\tilde \bff'(\xi)$
\end{enumerate}
\end{definition}

Existence and uniqueness of a kinetic solution of   \eqref{e5.1}-\eqref{e5.1'} with the non-degeneracy condition
implied by \eqref{e3.ndeg-new} was established in \cite{GH}. The kinetic solution must coincide with the entropy solution since both are obtained as the a.s.\   limit in $L_\loc^1((0,T)\X\R^P)$ of the solutions of the regularized parabolic approximation. 

The kinetic solutions of the periodic problem  \eqref{e5.1}-\eqref{e5.1'}  satisfy the following contraction property 
obtained in \cite{DHV},
as a consequence of the doubling of variables method introduced in \cite{DV}. Observe that, since we are dealing with an additive noise, the inequality holds a.s., instead of in average, i.e., for the expected values of the norms. 

\begin{proposition}[{\em cf.} \cite{DHV}] \label{P:5.stabper} Let $v_1$ and $v_2$ be two kinetic solutions of  \eqref{e5.1}-\eqref{e5.1'}, with initial data $v_{01},v_{02}$. Then, a.s., for a.e.\ $t\in[0,T]$, we have
\begin{equation}\label{e5.stabper}
\int_{\bbT^P}|v_1(t,x)-v_2(t,x)|\,dx\le \int_{\bbT^P}|v_{01}(x)-v_{02}(x)|\,dx.
\end{equation}
\end{proposition}

We next establish a result which is the analogue of theorem~2.1 of \cite{Pv}, where the method of reduction to the periodic case was introduced.

\begin{theorem}\label{T:5.1} Let $v:\Om\X(0,T)\X\bbT^P\to\R$, be a periodic entropy solution of \eqref{e5.1}-\eqref{e5.1'}, where $v_0(y)$ is a trigonometric polynomial as in \eqref{e4.0'}. Let $y(x):=(\l_1\cdot x, \cdots, \l_P\cdot x)$. Then, there exists a set $Z\subset\R^P$ of total measure, that is,  $\R^P\setminus Z$ has $P$-dimensional Lebesgue measure zero, such that, for all  $z\in Z$, the function $u(t,x)=v(t,z+y(x))$ is a  $\BAP$-entropy solution of an initial value problem as \eqref{e1.1}-\eqref{e1.2} with initial function $u_0(x)=v_0(z+y(x))$ and noise functions $h_k(z+y(x))$. Moreover, $Z$ does not depend on $\om\in\Om$ and can be taken as
the same for all trigonometric polynomials $v_0(y)$ in a countable family $\mathcal{T}$ dense in $L^1(\bbT^P)$. 
\end{theorem}

\begin{proof} Except for the independence of $Z$ with respect to $\om\in\Om$, the proof is totally similar to the one of theorem~2.1 of \cite{Pv}, and we refer to the latter for the proof of the first part. We assume that $\F$ has a countable  basis and let $\{\gamma_\ell(\om)={\bf1}_{A_\ell},:\, \ell \in\N\}$ where $\{A_\ell\,:\,\ell\in\N\}$ is  a basis for $\F$.  Also, let us assume that $\eta\in{\mathcal E}_0$,  $v_0\in\mathcal{T}$ where ${\mathcal E}_0$ is a countable dense subset of ${\mathcal E}$,
$\mathcal{T}$ is a countable family of trigonometric polynomials dense in $L^1(\bbT^P)$.
Set $J(\om,z,t,x):=\sum_{k\in\N}h_k(z+y(x))\b_k(t,\om)$, $w(\om,z,t,x)=v(\om,t,z+y(x))-J(\om,z,t,x)$ and $u_0^z(x)=v_0(z+y(x))$.  Let $Z_\ell(v_0, \eta)\subset\R^P$ be the set of Lebesgue points $z\in\R^P$ of
\begin{multline}\label{e400.entropy}
I_\ell(v_0, \eta)=\int_\Om \gamma_\ell(\om) \int_0^T\int_{\R^N} 
 \left\{(\eta(w)-\eta(u_0^z))\varphi_t+\sum_{i=1}^N q_i(w,t,x)\po_{x_i}\varphi \right.\\
 \left.+\sum_{i,j=1}^N r_{ij}(w,t,x)\po_{x_ix_j}^2\varphi\right. \\
\left.-\sum_{i=1}^N\left(\eta'(w)b_{i,x_i}(w,t,x)-q_{i,x_i}(w,t,x)\right)\varphi
+\sum_{i,j=1}^Nr_{ij,x_i}(w,t,x)\po_{x_i}\varphi\right\}\,dx\,dt\\
 - \int_0^T\int_{\R^N}\eta''(w)\sum_{k=1}^K\left(\sum_{i=1}^N\left(\po_{x_i}\z_{ik}(w,t,x)-\z_{ik,x_i}(w,t,x)\right)\right)^2\varphi\,dx\,dt.
 \end{multline}
 where  $\varphi$ runs along a countable dense subset of $C_c^\infty([0,T)\X\R^N)$. We then define
$Z(v_0,\eta):=\bigcap_{\ell\in\N} Z_\ell(v_0,\eta)$, $Z=\bigcap_{v_0\in\mathcal{T}, \eta\in{\mathcal E}_0}Z(v_0,\eta)$. We can easily check that $Z$ satisfies the assertion of the theorem. 

\end{proof}
 
Together with Theorem~\ref{T:5.1} the following lemma is also a very important ingredient in  the method of reduction to the periodic case in \cite{Pv}. In the latter, the analogue of \eqref{e5.elem} below is derived from Birkhoff's ergodic theorem. Here we give a different proof which has the advantage to give the validity of the referred equation for all $z_0\in\R^P$.  

\begin{lemma}\label{L:5.elem} If $w\in L^1(\bbT^P)$, $z_0\in\R^P$, $y(x)=(\l_1\cdot x, \cdots,\l_P\cdot x)$,
$x\in\R^N$, then we may define the map $x\mapsto w(z_0+y(x))$ as a function in $\B_*^1(\R^N)$. Moreover, we have for the $\B^1$-norm of  this function
\begin{equation} \label{e5.elem}
\Medint_{\R^N}|w(z_0+y(x))|\,dx=\int_{\bbT^P}|w(y)|\,dy.
\end{equation}
In particular, the mapping $w(y)\mapsto w(z_0+y(x))$ is an isometric isomorphism between $L^1(\bbT^P)$ and $\B_*^1(\R^N)$. 
\end{lemma}

\begin{proof} Consider the elementary trigonometric functions $E_0(y)=1$ and
$$
E_j^\pm(y)=e^{\pm i2\pi y_j}:[-1/2,1/2]^P\to\C,\quad j=1,\cdots,P,
$$ 
which can be viewed as functions on the $P$-dimensional torus $\bbT^P$, by the usual identification 
of $[-1/2,1/2]^P$ with periodic conditions on the boundary and the $P$-dimensional torus.  
We have $E_j^\pm(z_0+y(x))=e^{\pm i2\pi({z_0}_j+ \l_j\cdot x)}$ which clearly belong to $\AP(\R^N)$, since they are indeed periodic with period $(2\pi/(\l_j)_1,\cdots,2\pi/(\l_j)_N)$, $j=1,\cdots,P$. 
Since the (complex valued) continuous periodic functions on $[-1/2,1/2]^P$, or $C(\bbT^P)$,  form a closed algebra generated by the elementary trigonometric functions $E_j^\pm(y)$, $j=0,1,\cdots,P$, and $\AP(\R^N)$ is also a closed algebra, it follows that for any (complex valued) continuous periodic function $F\in C(\bbT^P)$, $F(z_0+y(x))\in\AP(\R^N)$. Observe also that we have, concerning the mean-value of $F(z_0+y(x))$,
$$
\Me(F):=\lim_{R\to\infty}\frac1{R^N}\int_{C_R}F(z_0+y(x))\,dx=\int_{\bbT^P}F(y)\,dy,
$$
since this is true when $F$ is a trigonometric polynomial, that is, when $F$ is a finite linear combination of $E_0(y)$ and trigonometric exponentials of the type $E^{\bar k}(y)=e^{i 2\pi \bar k\cdot y}$, with $\bar k=(k_1,\cdots,k_P)\in\Z^P$,  and these  are dense in 
$C(\bbT^P)$ with respect to the uniform topology. In particular, for any continuous periodic 
$F:\bbT^P\to\C$, the 
$\B^1$-norm of  $F(z_0+y(x))$ verifies
\begin{equation}\label{e5.elem'}
\Medint_{\R^N} |F(z_0+y(x))|\,dx=\int_{\bbT^P}|F(y)|\,dy.
\end{equation}
Since $C(\bbT^P)$ is dense in $L^1(\bbT^P)$, we deduce that, given $w\in L^1(\bbT^P)$, 
we can find a sequence $w_n\in C(\bbT^P)$, $n\in\N$, with $w_n\to w$ in $L^1(\bbT^P)$ and, so,
$w_n(z_0+y(x))$ is a Cauchy sequence in $\B^1(\R^N)$. Therefore, there exists a $g\in\B^1(\R^N)$ such that $w_n(z_0+y(x))\to g$ in $\B^1(\R^N)$. We notice that this function $g$ does not depend on the specific sequence of functions $w_n\in C(\bbT^P)$ converging to $w$ in $L^1(\bbT^P)$. Indeed,
if $\tilde w_n$ is another sequence in $C(\bbT^P)$ with $\tilde w_n\to w$ in $L^1(\bbT^P)$, then, by \eqref{e5.elem'},
 \begin{equation*}
\lim_{n\to\infty}\Medint_{\R^N} |w_n(z_0+y(x))-\tilde w_n(z_0+y(x))|\,dx=\lim_{n\to\infty}\int_{\bbT^P}|w_n(y)-\tilde w_n(y)|\,dy=0,
\end{equation*}
and so $w_n(z_0+y(x))$ and $\tilde w_n(z_0+y(x))$ converge to the same limit in $\B^1(\R^N)$. 
We may denote, without ambiguity, $g(x):= w(z_0+y(x))$. Moreover, since \eqref{e5.elem} holds for $w_n$, it also holds for $w$. 

Finally, concerning the fact that the mapping $w(y)\mapsto w(z_0+y(x))$ is an isometric isomorphism between $L^1(\bbT^P)$ and $\B_*^1(\R^N)$, that this mapping is injective it is clear. The surjectivity follows from the fact that any $g\in \B_*^1(\R^N)$ may be approximated in $\B_*^1(\R^N)$ by trigonometric polynomials in $\AP_*(\R^N)$, $g^n(y(x))$ with $g^n\in C(\bbT^P)$ and $g^n$ converging in $L^1(\bbT^P)$ to some $w\in L^1(\bbT^P)$. This then proves that $g$ may be represented as $w(z_0+y(x))$, which implies the surjectivity of the mapping.

\end{proof}

The following corollary is useful in connection with Theorem~\ref{T:5.1}. 

\begin{corollary} \label{C:newcor} Let $v:\Om\X(0,T)\X\bbT^P\to\R$, be the periodic entropy solution of \eqref{e5.1}-\eqref{e5.1'} with $v_0\in L^1(\bbT^P)$. Let $Z$ be the set of total measure given by Theorem~\ref{T:5.1}. Let $z\in Z$ be  fixed and  $y(x)=(\l_1\cdot x,\cdots, \l_P\cdot x)$. Then,  the function $u(t,x)=v(t,z+y(x))$ is a $\BAP$-entropy solution of \eqref{e1.1}-\eqref{e1.2} with initial function $v_0(z+y(x))$ and noise functions $g_k^z(x)=h_k(z+y(x))$. 
\end{corollary}

\begin{proof} Indeed, from the last lemma it follows, if $v_0^\a(y)$ is a sequence of trigonometric polynomials in $\mathcal{T}$ 
approximating $v_0(y)$ in $L^1(\bbT^P)$, then  
\begin{equation} \label{e5.elem''}
\Medint_{\R^N}|v_0^\a(z+y(x))-v_0(z+y(x))|\,dx=\int_{\bbT^P}|v_0^\a(y)-v_0(y)|\,dy,
\end{equation}
and so $v_0^\a(z+y(x))\to v_0(z+y(x))$ in $\B^1(\R^N)$ as $\a\to\infty$.
Therefore, if $u^{\a,z}(t,x)=v^\a(t,z+y(x))$ is the $\BAP$-entropy solution  of \eqref{e1.1}-\eqref{e1.2} with 
$u^{\a,z}(0,x)=v_0^\a(z+y(x))$, according to Theorem~\ref{T:4.1}, and $u^z(t,x)$ is the corresponding solution with 
 initial function $u(0,x)=v_0(z+y(x))$,   using \eqref{e2.4}, we obtain that $u^{\a,z}\to u^z$ in $L^\infty((0,T); \B^1(\R^N))$, as $\a\to\infty$, 
a.s.\ in $\Om$. Again, since by \eqref{e3.new} the norm of the $u^{\a,z}$'s in $L^1(\Om; L^\infty((0,T);\B^1(\R^N)))$  are
uniformly bounded by a function in $L^1(\Om)$, we conclude by dominated convergence that 
\begin{equation}\label{e.newcor}
u^{\a,z}\to u^z \quad \text {in $L^1(\Om;L^\infty((0,T); \B^1(\R^N)))$}.
\end{equation}
Finally, using again  Lemma~\ref{L:5.elem}, we deduce that we must have $u^z(t,x)=v(t, z+y(x))$, where $v(t,y)$ is the entropy solution 
of \eqref{e5.1}-\eqref{e5.1'}.
\end{proof}

\subsection{The limit as $z\to0$}\label{SS:4.1}  In this subsection we consider the limit as $z\to0$ of the 
$\BAP$-entropy solutions $u^z(t,x)=v(t, z+y(x))$ given by Corollary~\ref{C:newcor} and show that they converge to a $\BAP$-solution of \eqref{e1.1}-\eqref{e1.2}. Observe that, since 
$\B^1_*(\R^N)\subset \B^1(\R^N)$ such $\BAP$-entropy solutions belong to $\B_*^1(\R^N)$.
Similarly, if a $L^1(\bbG_{N})$-semigroup solution is the limit in $L^1(\Om;L^\infty((0,T); L^1(\bbG_N)))$ of  $\BAP$-entropy solutions of \eqref{e1.1}-\eqref{e1.2} belonging a.s.\ to $L^\infty((0,T); \B_*^1(\R^N))$, then, a.s., it belongs to  $L^\infty((0,T); L^1(\bbG_{*N}))$. We then, henceforth, call such $L^1(\bbG_N )$-semigroup solutions  
$L^1(\bbG_{*N})$-semigroup solutions of \eqref{e1.1}-\eqref{e1.2}. 

For the discussion in this subsection we assume the non-degeneracy condition \eqref{e5.NDC2}-\eqref{e5.NDC3}, in Section~\ref{S:6}, to assure the improved regularity of the periodic entropy solutions proved in \cite{CPa}.  

\begin{theorem}\label{T:4.2}  Let $z_n\in Z$ be a sequence converging to 0 and let $u^n(t,x)=v(t,z_n+y(x))$ be the $\BAP$-entropy solution given by Corollary~\ref{C:newcor}, where 
$v(t,y)$ is the periodic entropy solution of \eqref{e5.1}-\eqref{e5.1'}, with initial function $v_0\in C(\bbT^P)$. Then, $u^n$ converges in $L^1(\Om;L^1((0,T); L_\loc^1\cap\B^1(\R^N)))$  to a $\BAP$-entropy solution of 
\eqref{e1.1}-\eqref{e1.2} with $u_0(x)=v_0(y(x))$, which then may be represented  as $u(t,x)=v(t,y(x))$. 

As a consequence, let $u_0\in L^1(\bbG_{*N})$, so  $u_0(x)=v_0(y(x))$ for some $v_0\in L^1(\bbT^P)$.   Then $u(t,x)=v(t,y(x))$ is the $L^1(\bbG_{*N})$-semigroup solution of \eqref{e1.1}-\eqref{e1.2}  where 
$v(t,y)$ is the periodic entropy solution of \eqref{e5.1}-\eqref{e5.1'}, with initial function $v_0$. 
\end{theorem}
  
  \begin{proof} {\em Step \#1.} Let us denote by ${\mathfrak Y}$ the mapping $v(y)\mapsto v(y(x))$ from $L^1(\bbT^P)\to \B_*^1(\R^N)$. By Lamma~\ref{L:5.elem}, ${\mathfrak Y}$ is an isometric isomorphism. For $s\in\R$, $q\ge1$, let us define ${\mathcal W}_*^{s,q}(\R^N):={\mathfrak Y}[W^{s,q}(\bbT^P)]$, and
  $$
  \|v(y(\cdot))\|_{{\mathcal W}_*^{s,q}(\R^N)}=\|v\|_{W^{s,q}(\bbT^P)}.
  $$
 The first part of the statement  is proved following the same steps as the proof of the existence of a  $\BAP$-entropy solution of \eqref{e1.1}-\eqref{e1.2} as the limit of a vanishing viscosity sequence of solutions to the parabolic approximation as it was done in Section~\ref{S:4}, with the following adaptations. Now, besides the sequence $u^n(t,x)$, we also consider the sequence $v^n(t,y):=v(t,z_n+y)$. Recall that $v^n(t,y)$ is the periodic entropy solution  of \eqref{e5.1}-\eqref{e5.1'} with initial function $v_0(z_n+y)$ and noise functions $h_k(z_n+y)$, $k\in\N$. We can proceed with the above mentioned compactness method along the usual steps, Kolmogorov's continuity, Prohorov's theorem, Skorokhod's representation theorem, etc., corresponding to Propositions~\ref{P:3.2'}, \ref{P:3.2''}, \ref{P:3.2'''}, etc., simultaneously for both $u^n$ and $v^n$. While the steps for the sequence $u^n$ are similar to those for the vanishing viscosity sequence, the same is true for the sequence $v^n$. We combine both procedures transferring the regularity results for $v^n$ over to $u^n$ through the map ${\mathfrak Y}$. 
 
 {\em Step \#2.} Thus, combining the corresponding Proposition~\ref{P:3.2'} for $u^n$ and $v^n$, we get $u^n\in C^\l([0,T]; W_\loc^{-2,r}\cap {\mathcal W}_*^{-2,r}(\R^N))$. Concerning the results corresponding to Proposition~\ref{P:3.2''} for both $u^n$ and $v^n$, they can be combined by defining  
$$
 {\mathcal X}:=L^1((0,T);L_\loc^1\cap\B_*^1(\R^N))
 \bigcap C([0,T]; W_\loc^{-2,r}\cap {\mathcal W}_*^{-2,r}(\R^N)).
 $$    
  In the proof of the tightness corresponding to Proposition~\ref{P:3.2''}, tranferring the regularity of $v^n$ to $u^n$, we can now define $K_R=K_R^u\cap K_R^v$, where $K_R^u$ is as $K_R$ in the proof of Proposition~\ref{P:3.2''} and 
  \begin{multline*}
  K_R^v:= \{u\in {\mathcal X}\,:\, \|u\|_{C^\l([0,T];{\mathcal W}_*^{-2,r}(\R^N))}\le R,
\,  \|u\|_{L^1((0,T);{\mathcal W}_*^{s,r}(\R^N))}\le R,\\
  \|u\|_{L^\infty((0,T);\B_*^2(\R^N))}\le R\}.
  \end{multline*}
  The procedures to prove the tightness of the laws of $u^n$ in ${\mathcal X}$ are then totally similar to those in the proof of Proposition~\ref{P:3.2''}. Then Proposition~\ref{P:3.2'''} and the subsequent content of Section~\ref{S:4} may be repeated with no change, and this way we conclude that the sequence $u^n$ converges in $L^1(\Om;L^1((0,T);L_\loc^1\cap \B_*^1(\R^N)))$ to the $\BAP$-entropy solution of \eqref{e1.1}-\eqref{e1.2}, with $u_0(x)=v_0(y(x))$, and by Lemma~\ref{L:5.elem} it may be represented as $u(t,x)=v(t,y(x))$.  Indeed, by Lemma~\ref{L:5.elem} we deduce that $v_0(z+y(x))\to v_0(y(x))$, as $z\to0$,  in $\B^1(\R^N)$. Moreover, using again Lemma~\ref{L:5.elem},  we have
  \begin{multline*}
 \bbE\int_0^T\!\!\!\Medint_{\R^N}|u^n(t,x)-v(t,y(x))|\,dx\,dt
 \\=\bbE\int_0^T\!\!\!\Medint_{\R^N}|v(t,z_n+y(x))-v(t,y(x))|\,dx\,dt\\
 =\bbE\int_0^T\int_{\bbT^P}|v(t,z_n+y)-v(t,y)|\,dy\,dt\to0,\quad\text{as $z_n\to0$},
 \end{multline*}
 where we also use the continuity of translations in $L^1(\bbT^P)$. Therefore, $u^n(t,x)\to v(t,y(x))$ 
 in $L^1(\Om\X[0,T]\X\bbG_N)$, and so $v(t,y(x))$ is the $\BAP$-entropy solution $u(t,x)$ of \eqref{e1.1}-\eqref{e1.2} with $u_0(x)=v_0(y(x))$.
 
 {\em Step \#3.} Concerning the final part of the statement, it is proved as follows. When $u_0\in \AP_*(\R^N)$, by Lemma~\ref{L:5.elem} and its proof, $u_0(x)=v_0(y(x))$, for some $v_0\in C(\bbT^P)$, and so by the first part of the statement, $u(t,x)=v(t,y(x))$ is a $\BAP$-entropy solution of \eqref{e1.1}-\eqref{e1.2}. On the other hand, if $u_0\in L^1(\bbG_{*N})$, by Lemma~\ref{L:5.elem}, $u_0(x)=v_0(y(x))$, for some $v_0\in L^1(\bbT^P)$, and, if $v_0^n\in C(\bbT^P)$ is a sequence of continuous functions on the torus converging to $v_0$ in $L^1(\bbT^P)$,
  then, as in the proof of Theorem\ref{T:3new},  the $\BAP$-entropy solutions with initial functions $u_0^n(x)=v_0^n(y(x))$, $u^n(t,x)=v^n(t,y(x))$,  converge in $L^1(\Om;L^\infty((0,T);L^1(\bbG_{*N})))$ to a 
  $L^1(\bbG_{*N})$-semigroup solution of \eqref{e1.1}-\eqref{e1.2}, which can be represented as 
  $u(t,x)=v(t,y(x))$.  
       
  \end{proof}
  
  As a consequence of Theorem~\ref{T:4.2} we have the following result establishing the contraction property of the $L^1(\bbG_{*N})$-semigroup solutions.
  
  \begin{proposition}[$L^1$-mean contraction property]~\label{P:5.new}    Let 
 $ u_1(t,x), u_2(t,x)$  be two $L^1(\bbG_{*N})$-semigroup solutions of \eqref{e1.1}-\eqref{e1.2} with  initial data $u_{01}, u_{02}\in L^1(\bbG_{*N})$. 
 Then,  a.s.,  for a.e.\ $t>0$, 
 \begin{equation}\label{e5.4AP'} 
 \int_{\bbG_N}|u_1(t)- u_2(t)|\,d\mm \le \, \int_{\bbG_N}|u_{01}-u_{02}|\,d\mm.
 \end{equation}
 \end{proposition}
 
 \begin{proof} Using the isometric isomorphism ${\mathfrak Y}: L^1(\bbT^N)\to L^1(\bbG_{*N})$, $v(y)\mapsto v(y(x))$, since, by Theorem~\ref{T:4.2}, $u_1(t,x)=v_1(t,y(x))$, $u_2(t,x)=v_2(t,y(x))$, where 
 $v_1,v_2$ are the periodic entropy solutions with initial data $v_{01}, v_{02}$, such that ${\mathfrak Y}v_{0i}=u_{0i}$, $i=1,2$, \eqref{e5.4AP'} follows immediately from the contraction property for periodic entropy solutions \eqref{e5.stabper}.
 
 \end{proof}

\section{Asymptotic Behavior}\label{S:6}

In this section we study the asymptotic behavior of the $L^1(\bbG_{*N})$-semigroup solution obtained in the last section.  Thus, we keep considering  the algebra
generated by $\{e^{\pm 2\pi i\l_\ell\cdot x}\,:\, \ell=0,1,2,\cdots, P\}$, with $\l_\ell\in\R^N$, $\l_0=0$,   where $\Lambda=\{\l_1, \l_2, \cdots,\l_P\}$ is a $\Z$-linearly independent set in $\R^N$, and we keep  denoting the closure of  this algebra  in 
the $\sup$-norm by $\AP_*(\R^N)$. For any $g\in\AP_*(\R^N)$, we have $\Sp(g)\in\G_\Lambda$, where the latter is the smallest additive group containing $\Lambda$. We also keep assuming, as in the last section,  that the noise functions satisfy $g_k\in\AP_*(\R^N)$, $k\in\N$. For $y(x)=(\l_1\cdot x,\cdots, \l_P\cdot x)$ we have that $g_k(x)=h_k(y(x))$, where $h_k(y)\in C(\bbT^P)$, $k\in\N$.

From \eqref{e5.4AP} we can define   the transition semigroup in $L^1(\bbG_{*N})$ associated with \eqref{e1.1}:
$$
P_t\phi(u_0)=\bbE(\phi(u(t))),\quad \phi\in \mathcal{B}_b(L^1(\bbG_{*N})),
$$
where $u(t)$ denotes the $L^1(\bbG_{*N})$-semigroup solution with initial data $u_0$ at time $t$, which, to be more precise, we will henceforth denote $u^{u_0}(t)$, and $\mathcal{B}_b(L^1(\bbG_{*N}))$ denotes the bounded Borel function on $L^1(\bbG_{*N})$.  We keep the notation and assumptions of Section~\ref{S:5}. 
 
A probability measure $\mu$ on $L^1(\bbG_{*N})$ is a said to be an invariant measure for $(P_t)$ if we have
\begin{equation*}
P_t^*\mu =\mu,\quad t\ge0,\quad
\text{where $\la P_t^*\mu, \phi\ra=\la\mu, P_t\phi\ra$, for all $\phi\in C(L^1(\bbG_{*N}))$.}
\end{equation*}
It can be easily checked that $P_t(u_0,\Gamma):=P_t\chi_\Gamma(u_0)$, $u_0\in L^1(\bbG_{*N})$, 
$\Gamma\in\mathcal{E}:=\mathcal{B}_b(L^1(\bbG_{*N}))$,  defines a Markovian transition function.

 Recalling the definition of ${\mathcal W}^{s,q}(\R^N)$ in Section~\ref{S:5}, let  $S\subset L^1(\bbG_{*N})$ be defined by
$$
S=\{u\in {\mathcal W}_*^{s,q}(\R^N)\,:\, \int_{\bbG_{*N}}u(x)\,dx=0\},
$$
where $W^{s,q}(\bbT^P)$ is the Sobolev space such that the kinetic periodic solutions obtained in \cite{GH} with initial data in $L^3(\bbT^P)$ belong to $L^1((0,T); W^{s,q}(\bbT))$,
according to \cite{CPa}.
More specifically, we also recall the decisive estimate  (4.21) from \cite{CPa},  
  for the kinetic periodic solution on
 $\bbT^P$,
 \begin{equation}\label{e5.DV}
 \bbE\|v\|_{L^1((0,T);W^{s,q}(\bbT^P))}
 \le k_0(\bbE\|v_0\|_{L^3(\bbT^P)}^3+1+T),
 \end{equation}
 for some $q>1$, where $k_0$ depends only on the data of the periodic problem, provided the non-degeneracy condition \eqref{e5.NDC3}, with \eqref{e5.NDC2}, recalled below, holds.
 
We then define, 
$$
\|u\|_S:=\|u\|_{{\mathcal W}_*^{s,q}(\R^N)}.
$$
We notice that $S$ is a subspace of $L^1(\bbG_{*N})$ and $\|\cdot\|_S$ is a norm. Indeed,  since $W^{s,p}(\bbT^P)$ is continuously embedded in $L^1(\bbT^P)$, we have that if $\|u\|_S=0$, then $v=0$ in $L^1(\bbT^P)$, which, in turn,  by Lemma~\ref{L:5.elem}, implies that $u=0$ in $L^1(\bbG_{*N})$. The other properties for a norm are obviously checked. Thus, $\|\cdot\|_S$ is a norm in $S$.

Let $S_R:=\{u\in S\,:\, \|u\|_S\le R\}$. We claim that $S_R$ is compact in $L^1(\bbG_{*N})$. Indeed,
given a sequence $u_\a\in S_R$, we can find $v_\a(y(x))$, with $\|v_\a\|_{W^{s,p}(\bbT^P)}\le R$ 
and $u_\a=v_\a(y(x))$ in $L^1(\bbG_{*N})$. By the compactness of the embedding 
$W^{s,p}(\bbT^P)\subset L^1(\bbT^N)$, we may find a subsequence $v_{\a_k}$ converging in $L^1(\bbT^P)$. Then,  by Lemma~\ref{L:5.elem}, $u_{\a_k}(x)=v_{\a_k}(y(x))$ converges in $L^1(\bbG_{*N})$  to  certain $u\in S_R$, which proves the compactness of $S_R$.

Let us define the probability measures 
$$
\la \mu_T,\phi \ra :=\frac1T\int_0^T P_t\phi (u_0) dt =\frac1T\int_0^T \la P_t^*\d_{u_0},\phi\ra dt, \quad \phi\in \B_b(L^1(\bbG_{*N})),
$$
where, for a Banach space $E$,  $\B_b(E)$ is the space of bounded Borel functions on $E$. 
  We next prove that the family of probability measures over $L^1(\bbG_{*N})$,  $\{\mu_T\}_{T>0}$, is tight, aiming to apply Prohorov's theorem (see, e.g., \cite{Bi}).
  \begin{proposition}\label{existence-m}
The family  $\{\mu_T\}_{T>0}$  of measures over $\B_b(L^1(\bbG_{*N}))$ is tight and relatively weakly compact. Hence, there is a subsequence $\mu_{T_k}$ and $\mu\in \M_1(L^1(\bbG_{*N}))$ such that $\mu_{T_k}\wto \mu$. 
\end{proposition}
\begin{proof}
  We suppose $u_0(x)=v_0(y(x))$ is a trigonometric polynomial and we let $v$ be the corresponding kinetic periodic solution  on 
 $\R^P$, with initial datum $v(0,y)=v_0(y)$,
 as in the discussion of Section~\ref{S:4}.
 Also, assume that
 $\phi\in \B_b(L^1(\bbG_{*N}))$ has support in $S_R^C=L^1(\bbG_{*N})\setminus S_R$, where $S_R$ is as above. 
 
 Thus, we have
 \begin{multline*}
 |\la  \mu_T,\phi\ra|=| \frac1T\int_0^T P_t\phi (u_0) dt|=| \frac1T\int_0^T \bbE\phi (u(t)) dt|\le \|\phi\|_\infty\frac1{R} \frac1T
 \bbE\int_0^T \|u(t)\|_S\,dt\\
 = \|\phi\|_\infty\frac1{R}\frac1T \bbE\int_0^T \|v(t)\|_{W^{s,q}(\bbT^P)}\,dt
 \le \frac{C\|\phi\|_\infty}{R},
 \end{multline*} 
where we have used Corollary~\ref{C:newcor} and \eqref{e5.DV}, which proves the desired tightness of $\mu_T$, $T>0$, and so, by Prohorov's therorem, there exists a subsequence $\{\mu_{T_k}\}_{k\in\N}$ and a probability measure $\mu$ over $L^1(\bbG_{*N})$ such that   $\mu_{T_k}\wto \mu$.     
\end{proof}
\begin{proposition} Assume condition \eqref{e5.NDC} holds, with 
$\iota^\vartheta(\d)$ defined by \eqref{e5.NDC1}.
 The measure $\mu$ obtained in Proposition~\ref{existence-m} is an invariant measure for the transition semigroup $P_t$.
\end{proposition}  
\begin{proof}   
For any $t\ge0$ and $\phi\in C_b(L^1(\bbG_{*N}))$, we have
 \begin{multline*}
 \la P_t^* \mu,\phi\ra =\la \mu, P_t\phi\ra =\lim_{n\to\infty}\la \mu_{T_n}, P_t\phi\ra\\
 =\lim_{n\to\infty}\frac1{T_n}\int_0^{T_n}P_s (P_t\phi)(u_0)\,ds= \lim_{n\to\infty}\frac1{T_n}\int_0^{T_n}
 P_{t+s}\phi (u_0) \,ds\\
 =\lim_{n\to\infty}\frac1{T_n}\int_t^{T_n+t} P_s\phi(u_0)\,ds=  \lim_{n\to\infty}\frac1{T_n}\int_0^{T_n} P_s\phi(u_0)\,ds\\
 -\lim_{n\to\infty}\frac1{T_n}\int_0^{t} P_s\phi (u_0)\,ds+\lim_{n\to\infty}\frac1{T_n}\int_{T_n}^{T_n+t} P_s\phi (u_0)\,ds\\
 = \lim_{n\to\infty}\frac1{T_n}\int_0^{T_n} P_s\phi(u_0)\,ds=\la \mu,\phi\ra,
 \end{multline*}
 which proves that $\mu$ is an invariant measure for \eqref{e1.1}. 
 \end{proof}

The just described procedure to obtain an invariant probability measure
 follows the classical Krylov-Bogolyubov method as described, e.g., in \cite{DZ}.


  Let us consider  the case where the initial datum is a trigonometric polynomial and $v$ is a solution of 
 the corresponding periodic problem such that $u(t,x)=v(t,y(x))$, as in Theorem~\ref{T:4.2}.
 Observe that, from condition \eqref{e5.NDC}, with 
$\iota^\vartheta(\d)$ defined by \eqref{e5.NDC1},
 then
 the flux function $\tilde \bff(v)=(\l_1\cdot \bff(v),\cdots,\l_P\cdot \bff(v))$ and the viscosity matrix
 $\tilde \abf(v)=(\tilde \s(v))^T\tilde \s(v)$, with $\tilde \s(v)$ being the $N\X P$ matrix written by its columns $\tilde \s(v)=[\tilde \s_1(v),\cdots,\tilde \s_P(v)]$,  $\tilde \s_i(v)=\l_i\s$, $i=1,\cdots,P$,  
  satisfy the condition, for $\tilde b=\tilde \bff'$ and
 $\tilde \iota^\vartheta(\ve)$  defined by
 \begin{equation}\label{e5.NDC2} 
\tilde\iota^\vartheta(\ve)=\sup_{\a\in\R,{\bf n}\in \Z^P}
 \int_\R\frac{\ve(\tilde \abf(\xi):\frac{\bf n}{|\bf n|}\otimes\frac{\bf n}{|\bf n|}+\ve)}
 {(\tilde \abf(\xi):\frac{\bf n}{|\bf n|}\otimes\frac{\bf n}{|\bf n|}+\ve)^2+\ve^{\nu}| \tilde b(\xi)\cdot \frac{\bf n}{|\bf n|}+\a|^2} \vartheta(\xi)\, d\xi,
 \end{equation}
 with $\vartheta$ as in \eqref{e5.NDC1}, 
\begin{equation}\label{e5.NDC3}
 \tilde\iota^\vartheta(\ve)\le c_1^\vartheta\ve^\k,
 \end{equation}
 for some $c_1^\vartheta>0$, $1<\nu<2$ and $0<\k<1$.

 We point out that the non--degeneracy condition \eqref{e5.NDC3} is a little different from the one in \cite{CPa}. Besides the fact that it is based on $\Z^P$, not on $\R^P$, as in \cite{CPa}, we introduce here the function $\vartheta$ in the definition of $\iota^\vartheta$ in \eqref{e5.NDC2}. This is necessary as, with our assumptions of Lipschitz continuity of the flux function $\mathbf{f}$ and of the viscosity matrix $\Abf$, the integral in \eqref{e5.NDC2} cannot converge without the presence of a weight function like $\vartheta$. However, we can still deduce estimate \eqref{e5.DV} with some modifications in the proof in \cite{CPa} as explained  in the appendix \ref{s:A}. 
 
 The following result establishes the uniqueness of the invariant measure. 

\begin{proposition}\label{P:Final} Assume that condition \eqref{e5.NDC} holds, with 
$\iota^\vartheta(\d)$ defined by \eqref{e5.NDC1}.
 Then the invariant measure $\mu$ of the transition semigroup $P_t$ is unique.   
\end{proposition}
 \begin{proof}
 The periodic kinetic solutions of \eqref{e5.1}-\eqref{e5.1'} are  kinetic solutions in the sense of \cite{CPa} and so,
any two of these solutions satisfy (by the last equation in section~5 of \cite{CPa})
\begin{equation}\label{e5.8}
\lim_{t\to\infty}\|v^1(t)-v^2(t)\|_{L^1(\bbT^P)}=0.
\end{equation}
{}From \eqref{e5.8} we obtain, for any two  $\BAP$-solutions with trigonometric polynomials as initial data and  trigonometric polynomials as noise coefficients in $\AP_*(\R^N)$, the equation   
\begin{equation}\label{e5.9}
\lim_{t\to\infty}\|u^1(t)-u^2(t)\|_{L^1(\bbG_N)}=0.
\end{equation}
 This, together with the contraction property \eqref{e2.4AP},  implies the uniqueness of the invariant measure.  Indeed, given $\phi\in \Lip(L^1(\bbG_{*N}))$ and $u_0\in \AP_*(\R^N)$, if $\mu$ is the invariant measure constructed by the above  Krylov-Bogoliubov's  argument, that is $\mu=\lim_{T_k\to\infty} \mu_{T_k}$, where $\mu_{T}=\frac1T\int_0^T P_t^*\d_{u_0}\,dt$, 
 we have 
 \begin{multline*}
 |\la\nu,\phi\ra-\la\mu_{T_n},\phi\ra|=\frac1{T_n}\int_0^{T_n}|\la P_t^*\nu,\phi\ra-\la P_t^*\d_{u_0},\phi\ra|\,dt\\
= \frac1{T_n}\int_0^{T_n}\left| \int_{L^1(\bbG_{*N})} (P_t\phi( v_0)-P_t\phi(u_0))\,d\nu(v_0) \right|\,dt\\
 \le  \frac1{T_n}\int_0^{T_n}\int_{L^1(\bbG_{*N})} C_\phi\, \bbE\| v(\cdot,t)- u(\cdot,t)\|_{L^1(\bbG_{*N})}\,d\nu(v_0) \,dt\\
\le \int_{L^1(\bbG_{*N})}  C_\phi \,\bbE\frac1{T_n}\int_0^{T_n} \| v(\cdot,t)-u(\cdot,t)\|_{L^1(\bbG_{*N})}\,dt\,d\nu(v_0)\to 0, 
\end{multline*}
where $v(\cdot,t),u(\cdot,t)$ are the $L^1(\bbG_{*N})$-entropy solutions associated with the initial data $v_0,u_0$, respectively. 
Hence, making $T_n\to\infty$, using \eqref{e5.9},  we conclude
$$
 |\la\nu,\phi\ra-\la \mu, \phi\ra|=0,
 $$
 and so
 \begin{equation}\label{e50.uni}
 \la\nu,\phi\ra=\la \mu, \phi\ra
 \end{equation}
 for all $\phi\in\Lip(L^1(\bbG_{*N}))$. Now, it is easy to extend \eqref{e50.uni} to all 
 $\phi\in\B_b(L^1(\bbG_{*N}))$: First for $\phi={\bf1}_F$ where $F$ is any closed subset of $L^1(\bbG_{*N})$ and then, by the regularity of the probability measures $\mu$ and $\nu$,  for 
 $\phi={\bf 1}_A$,  for any Borel set $A$, that is,  $\nu(A)=\mu(A)$ for all Borel sets of $L^1(\bbG_{*N})$, which  implies the uniqueness of the invariant measure  $\mu$ for \eqref{e1.1}.     
 \end{proof}

\appendix
 
\section{Regularity}\label{s:A}

As mentioned in Section~\ref{S:6} we need to assume a non-degeneracy condition, namely \eqref{e5.NDC3}, that differs slightly from the one in \cite{CPa}, which, however, still yields the regularity estimates that they prove, with only a few minor modifications in order to accommodate the weight function $\vartheta(\xi)=(1+|\xi|^2)^{-1}$. Indeed, following the proof of theorem~4.1 in \cite{CPa}, it suffices to make a small modification on the estimates on the term $u^0$ and $u^\flat$ of their decomposition of the periodic kinetic solution (see equation (4.7) in \cite{CPa}) of the parabolic-hyperbolic equation that they consider. Since the estimates on both $u^0$ and $u^\flat$ are similar, we only point out the changes on the first one.

The non--degeneracy condition comes into play on page 982, when estimating the term $\widehat{u^0}(k,t)$, where the Cauchy-Schwarz inequality is used in order to make appear the integral that defines the function $\eta(\lambda)$ (cf. condition (4.1) in their paper), which corresponds to the function $\iota^\vartheta$ in \eqref{e5.NDC2} above. At this point, it suffices to multiply and divide by $\vartheta(\xi)^{-1/2}$ (i.e. by $(1+|\xi|^2)^{1/2}$) before applying the Cauchy-Schwarz inequality as shown below, in order to make appear the function $\iota^\vartheta$, instead of their function $\eta$:
\begin{align*}
\int_0^T|\widehat{u^0}(t,k)|^2 dt &= \frac{4}{|k|}\int_{-\infty}^\infty\left|\int\frac{\mathcal{A}|k|+\omega_k}{(\mathcal{A}|k|+\omega_k)^2+|F'(\xi)\cdot\widehat{k}+\tau|^2}\widehat{\chi^u}(\xi,k,0)\, d\xi\right|^2\, d\tau\\
 &\le \frac{4}{|k|}\int_{-\infty}^\infty\left(\int|\widehat{\chi^u}(\xi,k,0)|^2\vartheta(\xi)^{-1}\frac{\mathcal{A}|k|+\omega_k}{(\mathcal{A}|k|+\omega_k)^2+|F'(\xi)\cdot\widehat{k}+\tau|^2}\, d\xi\right)\\
&\quad \times\left(\int \frac{\mathcal{A}|k|+\omega_k}{(\mathcal{A}|k|+\omega_k)^2+|F'(\xi)\cdot\widehat{k}+\tau|^2}\vartheta(\xi) d\xi \right)\, d\tau\\
&\le \frac{4}{|k|\omega_k}\int\widehat{|\chi^u}(\xi,k,0)|^2\vartheta(\xi)^{-1}\left(\int_{-\infty}^\infty\frac{\mathcal{A}|k|+\omega_k}{(\mathcal{A}|k|+\omega_k)^2+|F'(\xi)\cdot\widehat{k}+\tau|^2}\, d\tau\right)d\xi\\
&\quad \times\sup_{\tau}\int \frac{\omega_k(\mathcal{A}|k|+\omega_k)}{(\mathcal{A}|k|+\omega_k)^2+|F'(\xi)\cdot\widehat{k}+\tau|^2}\vartheta(\xi) d\xi.\\
\end{align*}
 
Here, the rest of the argument in \cite{CPa} can be followed line by line, carrying the function $\vartheta(\xi)^{-1}$ multiplying $|\widehat{\chi^v}(0,k,\xi)|^2$ to deduce using the non--degeneracy condition that
\[
\int_0^T|k|^{1+\kappa}\omega_k^{1-\kappa}|\widehat{u^0}(k,t)|^2dt\le C\int |\widehat{\chi^v}(\xi,k,0)|^2\vartheta(\xi)d\xi,
\]
and summing over all frequencies $k$ yields
\[
\int_0^T\|u\|_{H_y^{(1-\alpha)\kappa+\alpha}}^2dt\le C(1+\|u_0\|_{L^3(\mathbb{T}^P)}),
\]
as in \cite{CPa}. 

As mentioned above, the same modification can be made to include the function $\vartheta$ in the estimate of $v^\flat$. In this case, we also have to use the integrability properties of the periodic kinetic solution, which is also an important point of the regularity analysis in \cite{GH}.  In summary, this is how we obtain \eqref{e5.DV}.

\end{document}